\documentclass[12pt,a4paper]{article}

\usepackage[utf8]{inputenc}
\usepackage[T1]{fontenc}
\usepackage[english]{babel}
\usepackage{graphicx}
\graphicspath{{Figures/}}

\usepackage[dvipsnames]{xcolor}
\usepackage{pgf,tikz}
\usetikzlibrary{arrows}
\usetikzlibrary{decorations.markings,decorations.pathmorphing,patterns}
\usepackage{pgfplots}
\pgfplotsset{compat=1.18}

\usepackage[format=hang,font=small]{caption}
\usepackage{booktabs,bm,multirow,dsfont,mathtools}
\usepackage{enumitem}
\usepackage{hyperref}

\usepackage{geometry}
\usepackage{amsmath,amsfonts,amssymb,amsthm}
\usepackage{dsfont,multicol,marvosym}
\usepackage{subcaption}

\renewcommand{\phi}{\varphi}
\renewcommand{\theta}{\vartheta}
\renewcommand{\epsilon}{\varepsilon}
\newcommand{\field}[1]{\mathbb{#1}} 
\newcommand{\R}{\field{R}}

\newcommand{\B}{\field{B}}
\newcommand{\A}{\field{A}}
\newcommand{\Id}{\field{I}}
\newcommand{\V}{\field{V}}
\newcommand{\M}{\field{M}}

\newcommand{\EE}{\mathcal{E}}
\newcommand{\VV}{\mathcal{V}}

\DeclareMathOperator{\diag}{diag}

\DeclareMathOperator{\tr}{tr}
\newcommand{\sh}[1]{{#1}_\mathrm{sh}}
\newcommand{\Ash}{\sh\A}

\newcommand{\abs}[1]{\left\lvert#1\right\rvert}                  
\newcommand{\norm}[1]{\left\lVert#1\right\rVert}    
\newcommand{\scal}[2]{\left\langle #1,#2\right\rangle}
\newcommand*{\dd}{\mathop{}\!\mathrm{d}}

\newtheorem{theorem}{Theorem}
\newtheorem{lemma}[theorem]{Lemma}
\newtheorem{prop}[theorem]{Proposition}
\newtheorem{corol}[theorem]{Corollary}
\theoremstyle{definition}
\newtheorem{defin}[theorem]{Definition}
\newtheorem{remark}[theorem]{Remark}
\newtheorem{example}[]{Example}

\numberwithin{equation}{section}
\numberwithin{theorem}{section}

\title{A Massera-type Theorem on relative-periodic solutions for a second-order model of rectilinear locomotion}

\author{Paolo Gidoni\footnote{Dipartimento Politecnico di Ingegneria e Architettura, Università degli Studi di Udine, Via delle Scienze 206, 33100 Udine, Italy. \Letter \;\texttt{paolo.gidoni@uniud.it}} ~and Alessandro Margheri\footnote{Centro de Matemática, Aplicações Fundamentais e Investigação Operacional, Departamento de Matemática, Faculdade de Ciências, Universidade de Lisboa, Campo Grande, Edifício C6, piso 2, 1749-016 Lisboa, Portugal. \Letter \;\texttt{amargheri@fc.ul.pt}}}

\date{}

\begin{document}

\maketitle

\begin{abstract}
We study the existence of a global periodic attractor for the reduced dynamics of a discrete toy model for rectilinear crawling locomotion, corresponding to a limit cycle in the shape and velocity variables. The body of the crawler consists of a chain of point masses, joined by active elastic links and subject to smooth friction forces, so that the dynamics is described by a system of second order differential equations. Our main result is of Massera-type, namely we show that the existence of a bounded solution implies the existence of the global periodic attractor for the reduced dynamics. In establishing  this result, a contractive property of the dynamics of our model plays a central role. We then prove  sufficient conditions  on the friction forces for  the existence of a bounded solution, and therefore   of the attractor. We also provide  an example showing that,  if we consider more general friction forces, such as  smooth approximations of dry friction,   bounded solutions may not exist.
\end{abstract}


\section{Introduction}

 Many mechanical systems  are modeled by time-periodic ordinary differential equations in $\R^d$: consider, for instance, systems subject to a periodic forcing.  Compared to autonomous systems, for which a qualitative analysis naturally starts from characterizing equilibria, in time-periodic systems a first fundamental step is often the more challenging search for periodic solutions and their stability properties.
 
 A  relevant contribution on this topic is due to Massera \cite{Massera} in 1950.  He proved that,  in the scalar case    $(d=1)$, if there exists   a solution of  a $T$-periodic equation which is well defined  and bounded in an interval of the form $[t_0, +\infty)$,    then such solution converges to a $T$-periodic   solution. 
In the same paper Massera proved that this result is valid also in the planar case $d=2,$ under the mild (but sharp) additional  assumption that all  solutions are globally defined in future. For a nice recent presentation of Massera's theorems see \cite{Or}.   An analogous of Massera's theorem was obtained for  a special class of time periodic ordinary differential equations in  $\R^d$  for arbitrary $d$  in \cite{Smith}.  This class   satisfies  an additional  assumption,    expressed in terms of  a suitable  Lyapunov-like function.

Results that infer the existence of a periodic (or quasiperiodic)  solution given the existence of a bounded one have been often referred to as  \emph{Massera-type} theorems. These theorems have been proposed in various settings, including functional   functional differential equations \cite{ChowHale,ShNa},  generalized differential equation \cite{FlMeSl},  PDEs    \cite{Capistrano2023,LCLL99} and differential inclusions \cite{GudMak}.

In this paper,   we obtain a Massera-type result for the dynamics arising from a toy-model of crawling locomotion, namely Theorem \ref{th:Massera_red}.
The model we consider,  is illustrated in Figure~\ref{fig:crawler} and consists of a chain of $n$ blocks on a line, with adjacent blocks being joined by elastic active (i.e~changing rest length) links and friction forces opposing the motion of each block. 
The time-periodicity in the governing equations is associated with the \emph{gait} adopted, i.e.~the periodic pattern of ``actions'' performed by the crawler to move. More precisely, in our model, the gait consists of a periodic change in rest lengths of the links and possibly in the friction forces, which indirectly leads to a change in the shape of the crawler (which we will show to be only \emph{asymptotically} periodic). Equivalently, each active link can be thought as an actuator with controlled periodic shape, connected to the blocks by an elastic element, acting analogously to a muscle-tendon system.

Simpler models without elasticity, so that each gait coincides with a prescribed periodic shape change of the crawler has also been considered \cite{BPZZ16,FigKny,GiMaRe}. We notice, however, that crawling animals, such as earthworms or inchworms, are generally soft bodied. Moreover, in recent years, the field of soft robotics has shown the most interest in these locomotion strategies (see \cite{CaPiLa} and the introduction in \cite{Gid18} for a quick discussion), where a deformable body is the principal component. Both these considerations make the elastic case worth to study. 

The key feature of studying a locomotion model is that periodicity is not  anymore what we expect asymptotically, since it would mean the inability of the gait to produce any advancement.
A proper way to address  this issue comes from \emph{geometrical mechanics}, with the notion of \emph{relative-periodicity}. The concept applies to flows on a manifold $M$ which are equivariant with respect to the action of a Lie Group $G$, so that a reduced dynamics can be properly defined on the quotient $M/G$. An orbit on $M$ is relative-periodic if its projection on $M/G$ is periodic. This means that a relative-periodic solution with initial state $s_0\in M$ reaches, after a period, a state $\Gamma_g s_0$, where $\Gamma_g$ describes the action of an element $g\in G$.

This structure becomes simpler and intuitive in our basic model of crawling. The manifold is the phase space $\R^{2n}$ and the symmetry group $G\cong\R$ consists of the rigid translations of the body along the line, since elastic forces depend only on the distances between the blocks and not on their absolute positions along the line. Relative-periodicity corresponds to periodicity for a reduced $\R^{2n-1}$ dynamics which describes the shape the crawler (i.e. relative distances between the blocks) and the velocity of each block. However, a relative-periodic motion of the crawler is not necessarily a periodic one when seen on the whole phase space $\R^{2n}$: after a cycle the crawler might have advanced of a certain distance $g\in \R$, so that the final state is a rigid translation by $g$ of the initial one, with which it shares shape and velocities.

The link between relative-periodicity and locomotion is well known \cite{KelMur}. The most famous examples, however, usually consider simplified models where, for any given time-periodic input, all solutions of the systems are relative-periodic: we mention, for instance, kinematic descriptions of  wheeled robots or models for swimming at low Reynolds number \cite{FaPaZo,KelMur}. Yet, in most concrete situations relative-periodicity emerges only as an asymptotic behaviour of the system. Usually this relates with the relevance of inertia and/or elasticity in the equations governing locomotion, making an initial velocity and/or deformation incompatible with any relative-periodic solution of the system. In such situation what we might expect is a relative-periodic Massera type result: that is, every solution whose reduced dynamics on $M/G$ is bounded admits a limit cycle in $M/G$. In our case, this will mean that every solution bounded in shape and velocity converges to a relative-periodic solution.

For crawling models analogous to the one of Figure~\ref{fig:crawler}, this type of asymptotic behaviour has been studied in \cite{BPZZ16,FigKny,GiMaRe,WagLau} assuming prescribed shape or in \cite{ColGidVil} at a quasistatic regime with dry friction. Both situations significantly reduce the dimension and the complexity of the system; in particular with prescribed shape the dynamics is described simply by a scalar first order differential equation.  

In this paper we study instead the full system, where both  elastic links and the inertia of the blocks are considered.
The complication added by this framework lies not only in the  higher dimensionality of the system, but also on the fact that, compared to the other cases, monotonicity properties of the friction forces are present only on a subspace, so they cannot be directly employed to obtain a contraction-type structure of the solutions.

Extending our view also to other locomotion models, the convergence to a limit relative-periodic behaviour has been observed for other models with prescribed shape, e.g.~\cite{PoFeTa,SanZop}, or for quasistatic models with an elastic body, studied only in a small-actuation limit perturbing the autonomous fully-passive system \cite{ADGZ17,Levillain24,MonDeS,PasOr,TINK12}. However, up to our knowledge, a full second-order dynamics with inertia and an elastic body has been considered until now only in \cite{EldJac}, where a system of four masses in a tetrahedral structure  with elastic links were studied, again with perturbative methods in the small-actuation limit. In this work, we prove instead our convergence results directly for the time-periodic system, allowing large deformations.

Massera-type results  leave open  two relevant questions, which we also address for our model: whether all bounded solutions converge to the same unique limit cycle and whether a bounded (on $[t_0,+\infty)$) solution exists to start with.

  The uniqueness of the limit cycle in models of  crawling seems to be at least partially related to  the strong monotonicity of the friction force-velocity laws \cite{GiMaRe}. In this paper, such a property is assumed in hypothesis \ref{hyp:A4} and uniqueness of the limit behaviour is  provided jointly with our Massera-type Theorem. However, strong monotonicity fails in some other relevant scenarios not considered here, e.g.~for dry friction, and multiplicity of the limit behaviour has been observed in such a case for simplified versions of our model. Precisely, multiplicity of the limit velocity has been shown for a dynamic model with prescribed shape \cite{GiMaRe}, whereas multiplicity of the limit cycle in the shape component can be found in a quasistatic setting \cite{ColGidVil}.

   Regarding the existence of a bounded solution,  such property is often obtained by finding  a  suitable Lyapunov-like  or guiding function which expresses   some dissipation mechanism of the dynamical system. In these cases what is found is actually  the existence of  a bounded set in which all the solutions of  our system   eventually enter.  This property of a dynamical system is called point  dissipativity \cite{Hale80,levinson44}.
   Notice that when a contractive structure of  solutions is present, as it is the case for our reduced dynamics,  point dissipativity is  actually  equivalent to the existence of a solution bounded in future, instead of just a sufficient condition.
     
  In this paper we prove the point dissipativity of the reduced dynamics in Lemma~\ref{lemma:stiff_dissip} under the additional assumption \ref{hyp:stiffbody}, which introduces some symmetry in the model and requires a sufficiently stiff body. We also observe that boundedness can be easily deduced  in the case of time-dependent viscous friction, cf.~Theorem~\ref{th:attractor_viscous}. 
  
However, such  property does not hold  in general  for our model of locomotion. In fact, in Section~\ref{sec:massera} we present an example of a system for which the reduced dynamics is unbounded (see Example~\ref{ex:resonance}).  The friction forces considered  in the example are bounded, so that they  do not verify \ref{hyp:A4},  and there is an internal resonance which drives the  the oscillations in the shape  to infinity.  
This leaves open the question  whether or not general  strongly monotone friction forces overcome internal resonances, thus leading to a global limit cycle.

The  paper is structured as follows.  In  Section~\ref{sec:model}, we  introduce our model of rectilinear locomotion  and its  main structural assumptions, showing the global forward existence of solutions in  Theorem \ref{th:glob_existence}. The  definitions of relative-periodic solution  and reduced dynamics for our rectilinear model  are discussed   in Section~\ref{sec:relative}. 
In Section~\ref{sec:massera} we establish a   Massera-type theorem for the reduced dynamics (Theorem \ref{th:Massera_red}),  and give the corresponding  result  for the locomotion model (Corollary \ref{cor:Massera_loc}).  We also establish the contractive property of the reduced dynamics (Theorem \ref{th:contrazione} and Corollary \ref{cor:contrazione_lip}). This section concludes  with the above mentioned example of divergent behaviour (Example~\ref{ex:resonance}).
Section~\ref{sec:bound} discusses the results about point dissipativity  of our model in the case of viscous friction (Theorem \ref{th:attractor_viscous}) or under  suitable additional assumptions (equal masses,  equal friction forces and sufficient stiffness  of the links, see  Theorem \ref{th:attractor_stiff}). Proofs of these results are provided in Section~\ref{sec:viscous} and Section~\ref{sec:stiff}, respectively.

%%%%%%%%%%%%%%%%%%%%%%%%%%%%%%%%%%%%%%%%%%%%%%%%%%%%%%%%%%%%%%

\section{A model of rectilinear crawling locomotion} \label{sec:model}
\begin{figure}
	\centering
	\centering
	\begin{tikzpicture}[line cap=round,line join=round,x=4mm,y=4mm, line width=1pt]
		\clip(-3,-2) rectangle (33,7);
		\draw [line width=1pt, fill=gray!40] (0,0.5)-- (3,0.5)--(3,3.5)-- (0,3.5)-- (0,0.5);
		\draw[decoration={aspect=0.5, segment length=1.5mm, amplitude=1.5mm,coil},decorate] (3,2)-- (5.2,2.);
		\draw[decoration={aspect=0.5, segment length=1.5mm, amplitude=1.5mm,coil},decorate] (6.8,2)-- (9.,2.);
		\draw (5.2,2)--(6.3,2);
		\draw (6.5,2)--(6.8,2);
		\draw (6.3,1.6)-- (5.5,1.6)--(5.5,2.4)--(6.3,2.4);
		\draw (5.7,1.8)-- (6.5,1.8)--(6.5,2.2)--(5.7,2.2);
		\draw [line width=1pt, fill=gray!40] (9,3.5)-- (9,0.5)-- (12,0.5)-- (12,3.5)-- (9,3.5);
		\draw[decoration={aspect=0.5, segment length=1.5mm, amplitude=1.5mm,coil},decorate] (12,2)-- (14.2,2.);
		\draw[decoration={aspect=0.5, segment length=1.5mm, amplitude=1.5mm,coil},decorate] (15.8,2)-- (18.,2.);
		\draw (14.2,2)--(15.3,2);
		\draw (15.5,2)--(15.8,2);
		\draw (15.3,1.6)-- (14.5,1.6)--(14.5,2.4)--(15.3,2.4);
		\draw (14.7,1.8)-- (15.5,1.8)--(15.5,2.2)--(14.7,2.2);
		\draw [line width=1pt, fill=gray!40] (18,3.5)-- (18,0.5)-- (21,0.5)-- (21,3.5)-- (18,3.5);
		\draw[decoration={aspect=0.5, segment length=1.5mm, amplitude=1.5mm,coil},decorate] (21,2)-- (23.2,2.);
		\draw[decoration={aspect=0.5, segment length=1.5mm, amplitude=1.5mm,coil},decorate] (24.8,2)-- (27.,2.);
		\draw (23.2,2)--(24.3,2);
		\draw (24.5,2)--(24.8,2);
		\draw (24.3,1.6)-- (23.5,1.6)--(23.5,2.4)--(24.3,2.4);
		\draw (23.7,1.8)-- (24.5,1.8)--(24.5,2.2)--(23.7,2.2);
		\draw [line width=1pt, fill=gray!40] (27,3.5)-- (27,0.5)-- (30,0.5)-- (30,3.5)-- (27,3.5);
		\draw [thick] (-2,0.5)-- (32,0.5);
        \fill [pattern = north east lines] (-2,0) rectangle (32,0.5);
		\draw (1.5,2) node[] {$m_1$};
		\draw (10.5,2) node[] {$m_2$};
		\draw (19.5,2) node[] {$m_3$};
		\draw (28.5,2) node[] {$m_4$};
	\end{tikzpicture}
	\caption{The toy model of rectilinear crawling locomotion}
	\label{fig:crawler}
\end{figure}
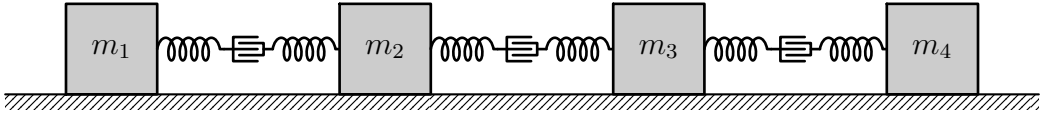

We now consider the  mechanical  system  represented in Figure~\ref{fig:crawler}, consisting of a chain of blocks on a line. Each block has mass $m_i>0$ and position along the line $x_i\in\R$, for $i=1,\dots,n$ (more generally, $x_i$ should be seen as the displacement of the block with respect to a given reference configuration, cf.~Remark~\ref{rem:displacement}. Thus the state of the system is described by the column vector  $x=(x_1,\cdots, x_n)^\top\in \R^n,$ where the $(\cdot)^\top$ denotes the transposition operation. In what follows, for sake of simplicity, if $x$ and $v$ are are column vectors,  we  denote by  $(x,v)$ the column vector $(x^\top, v^\top)^\top$.
We introduce  the inertia matrix   $\M\in\R^{n\times n}$
\begin{equation}\label{eq:massmatrix}
	\mathbb{M}=\diag(m_1,m_2,\cdots,  m_n),\quad \,m_i>0,
\end{equation}
\begin{equation*}
\norm{\cdot}_{\mathbb{M}}\coloneqq \langle{\mathbb{M}}\cdot, \cdot\rangle,
\end{equation*}
where $\langle\cdot,\cdot\rangle$ denotes the standard euclidean inner product in $\R^k$, with associated norm   $\norm{\cdot}$.  
The evolution of our model  is  governed  by following system of second order ODEs:
\begin{equation}\label{eqcrawl}
	{\mathbb{M}}\ddot x+F(t,\dot x)+\nabla_x\EE(t,x)=0
\end{equation}
where $\nabla_x$ denotes the gradient with respect to the $x$ variable, here represented for convenience as a column vector, and the dot $\dot{\ }$ denotes the (total) time-derivative $\frac{\dd}{\dd t}$.

In the model two types of forces are considered: friction forces, described   by $-F(t,\dot x)$, and elastic forces, described by  $-\nabla_x\EE(t,x)$.

The vector function $F\colon \R\times \R^n\to \R^n$ is of the form  $F(t,v)=(F_1(t,v_1),\cdots, F_n(t,v_n))^\top$. This means that on each block is acting a friction force $-F_i=-F_i(t,v_i)$, where $v_i$ is the velocity of the $i$-th block. Please notice that we have introduced friction using functions $F_i$, so that the actual friction force is $-F_i$: we made this choice so that \eqref{eqcrawl} resembles the custom notation for scalar damped oscillators.
 We make the following assumptions on $F$:
\begin{enumerate}[label=\textup{(A\arabic*)}]
	\item \label{hyp:A1} $F\colon \R\times \R^n\to\R^n$ is a Carathéodory function\footnote{We recall that a function $f=f(t,u)\colon \R\times \R^d\to \R$ satisfies is a  \emph{Carathéodory function} if the following conditions hold:
		\begin{itemize}
			\item for every $u\in\R^d$ the function $f(\cdot,u)$ is measurable in $t$;
			\item for almost every $t\in\R$ the function $f(t,\cdot)$ is continuous in $u$;
			\item for every compact set $K\subset \R\times \R^d$ there exists a Lebesgue integrable function $b_K(t)$ such that $\abs{f(t,u)}\leq b_K(t)$ for every $(t,u)\in K$.
	\end{itemize}}  and is  $T$-periodic in $t$;
	\item \label{hyp:A2} $F(t,v)$ is locally  Lipschitz continuous in $v$ uniformly in $t\in[0,T]$, namely, for every compact set $K\subset \R^n$ there exist a constant $\Lambda_K$ such that for every $t\in[0,T]$ and $v,w\in K$ it holds $\norm{F(t,v)-F(t,w)}\leq \Lambda_K \norm{v-w}$;
	\item \label{hyp:A3} $F_i(t,0)=0$ for every $t\in \R$ and $i=1,\dots,n$;
	\item \label{hyp:A4} there exists $\hat \mu>0$ such that
	\begin{equation*} 
		\langle F(t,v)-F(t,w),v-w\rangle\, \geq \hat \mu\norm{v-w}^2 \qquad \text{for every $(t,v,w) \in\R\times \R^n\times \R^n$.}
	\end{equation*}
\end{enumerate}
The first two are classical regularity assumptions, that can be equivalently reformulated by requiring  the analogous properties  on each component. The third one comes from the structure of friction forces: namely the force acting on a block is zero if the velocity of the block is zero. Assumption \ref{hyp:A4}  gives the strong monotonicity of each $F_i$, which combined with \ref{hyp:A3} guarantees  that friction forces are always opposing motion.

Elastic forces are, as usual, introduced as $-\nabla_x\EE(t,x)$, where  $\EE(t,x)$ is the elastic energy.  Notice that we are considering a \emph{time-dependent} elastic energy $\EE(t,x)$: it will account not only for deformations from the rest configuration, but also for the actuation on the system, described as a time-dependent change in the rest configuration.  

The peculiarity of locomotion models is that the elastic energy does not depend on the whole configuration space, but can be seen as defined on the subspace identifying the shape of the locomotor. To be more precise, let us introduce the projection matrix
$P\in\R^{(n-1)\times n}$   given by  
\begin{equation*}
P\coloneqq \begin{pmatrix}
	-1 & 1 & 0& 0&\cdots& 0 & 0\\
	0 & -1 & 1& 0&\cdots& 0 & 0\\
	\vdots &\vdots & \vdots  &\vdots&\cdots &\vdots  &\vdots\\
	0& 0 &0 &0&\cdots&-1&1
\end{pmatrix}
\end{equation*}
The matrix $P$ associates to each states $x\in \R^n$ the corresponding shape vector $z\in\R^{n-1}$  defined  by 
\begin{equation*}
	z=P x\in\R^{n-1}
\end{equation*}
with components  $z_i=x_{i+1}-x_i$, for $i=1,\dots,n-1$.
We assume that
\begin{enumerate}[label=\textup{(A\arabic*)},resume]
\item  \label{hyp:A5}	the elastic energy $\EE(t,x)\colon\R\times\R^n\to [0,+\infty)$  has the form
\begin{equation*}
	\EE(t,x)=\frac{1}{2}\norm{z-L(t)}_{\Ash}^2=\frac{1}{2}\norm{P x-L(t)}_{\Ash}^2
\end{equation*}
where $\Ash\in\R^{(n-1)\times (n-1)}$ is a symmetric positive-definite matrix, defining the norm $\norm{\cdot}_{\mathbb{\Ash}}\coloneqq \langle{\mathbb{\Ash}}\cdot, \cdot\rangle$ on $\R^{n-1}$,  and 
$L\colon\R\to\R^{n-1},$ $L(t)=(L_1(t),\cdots, L_{n-1}(t))^\top$   is a $T$-periodic  function belonging to $L^\infty([0,T])$.
\end{enumerate}

 It will be convenient to express the elasticity matrix as an operator on the whole space of configurations $\R^n$, instead of on the smaller shape subspace. To do so, we set $\A\coloneqq P^\top \Ash P\in \R^{n\times n}$. We note   that  $\A$ is positive semidefinite since it has rank $n-1$. 
  Hence, we have
\begin{equation} \label{eq:Aenergy}
\nabla_x\EE(t,x)=\A x -\ell(t) \qquad\text{where\, $\ell(t)= P^\top\Ash L(t)$}
\end{equation}

\begin{example}
The simplest example of this structure is the one shown in Figure~\ref{fig:crawler}, where we have $n-1$ active elastic links, all having the same elastic constant $k>0$, joining each couple of consecutive blocks. In such a case we have $\Ash=k\Id_{n-1}$, where $\Id_q$ denotes the $q\times q$ identity matrix, and
$L_i(t)$ is the rest length of the $i$-th link. Thus, the terms in the gradient $\nabla_x\EE(t,x)$ in \eqref{eq:Aenergy} take the form
\begin{align*}
\A=k\,\begin{pmatrix} 
		1  & -1 & 0 & \cdots & 0 & 0\\
		-1 &  2 & -1& \cdots & 0 & 0\\
		0  & -1 & 2 & \cdots & 0 & 0\\
		\vdots & \vdots & \vdots & \ddots & \vdots & \vdots \\
		0 & 0 & 0 & \cdots & 2 & -1\\
		0 & 0 & 0 & \cdots & -1 & 1\\
	\end{pmatrix}\,,
&&
\ell(t)=k\,\begin{pmatrix}
	-L_1(t)\\
	-L_2(t)+L_1(t)\\
	-L_3(t)+L_2(t)\\
	\vdots\\
	-L_{n-1}(t)+L_{n-2}(t)\\
	L_{n-1}(t)
\end{pmatrix}
\end{align*}
A general positive definite matrix $\Ash$ allows us to consider more general situation, where the links have different elasticity constants $k_i$, or where more complex structures are present, with links joining also non adjacent blocks.
\end{example}

\smallskip

Under   assumptions  \ref{hyp:A1} \ref{hyp:A2}, \ref{hyp:A3}, \ref{hyp:A5}, from  the general theory of ODEs \cite[Section 1.5]{Hale80}  it follows  that for any $(t_0, x_0, v_0)\in \R\times\R^n\times\R^n$
there exists a unique  maximal generalized solution of \eqref{eqcrawl} $t\to x(t)$ defined on an open interval $I$ containing $t_0$ such that  $(x(t_0), \dot{x}(t_0))=(x_0, v_0).$   We recall that a generalized solution $x(t)$ of \eqref{eqcrawl} is a function defined on a nondegenerate interval $I$  which belongs to $W^{2,1}_\mathrm{loc}(I, \R^n)$ (i.e.~a function with absolutely continuous derivative $v=\dot{x}$) which satisfies \eqref{eqcrawl}  almost everywhere in $I$. Throughout the paper, solutions of differential equations are to be intended in this generalized sense.
   
The existence in future of all solutions of system \eqref{eqcrawl} is guaranteed by the following result, where we replace \ref{hyp:A4} with the more general assumption  \ref{hyp:Dfric} on the friction forces. 

\begin{theorem}\label{th:glob_existence}
Assume that  \ref{hyp:A1} \ref{hyp:A2}, \ref{hyp:A3}, \ref{hyp:A5} and 
\begin{enumerate}[label=\textup{(A4$^*$)}] 
	\item\label{hyp:Dfric}  $\langle F(t,v),v\rangle\, \geq 0$,  for any $(t,v) \in\R\times \R^n$
\end{enumerate}
 hold.   Fixed  any $t_0\in\R$, the solutions of system \eqref{eqcrawl}  are defined on $[t_0,+\infty).$
\end{theorem}

The proof of Theorem~\ref{th:glob_existence} is given in Appendix~\ref{sec:globalexistence} and is based on an energy estimate, adapting to our situation the approach of \cite[Proposition 3.3]{GidRiv}. 

In what follows, without loss of generality, we assume $t_0=0$.

%%%%%%%%%%%%%%%%%%%%%%%%%%%%%%%%%%%%%%%%%%%%%%%%%%%%%%%%%%%%%%%%%%%%%%%%

\section{Relative-periodic solutions in rectilinear locomotion} \label{sec:relative}

Our dynamics \eqref{eqcrawl} can be equivalently formulated on the phase space as a system of $2n$ first order differential equations $(\dot x, \dot v)=\Phi(t, x,v)$, namely
 \begin{equation}\label{eq:2n_system}
	\begin{cases}
		\dot{x}= v\\
		\dot v=-\mathbb{M}^{-1}F(t,v)-\mathbb{M}^{-1}\A x +\mathbb{M}^{-1} \ell(t)
	\end{cases}
\end{equation} 
We notice that the vector field $\Phi$ on $\R^{2n}$ defined by \eqref{eq:2n_system} is invariant for translations of the crawler. Namely, writing $\mathds{1}_n=(1,1,\dots,1)^\top\in\R^n$ and $\Gamma_g(x,v)=(x+g\mathds{1}_n,v)$ for any $g\in \R$, we have 
\begin{equation}\label{eq:symmetryaction}
	\Phi(t, x,v)=\Phi(t, \Gamma_g (x,v))
\end{equation} 
for every $(t,x,v,g)\in\R\times\R^n\times\R^n\times \R$.

More precisely, we should say that at  each time $t$  the vector field $\Phi$ is equivariant with respect to the group action $\Gamma_g$ of the symmetry group $G=\mathrm{SE}(1)=\{g\in \R\}$, where $\mathrm{SE}(1)$ is the special Euclidean group on $\R$.
Since the  structures of $\R^{2n}$, seen as a manifold, and of the symmetry group  $\mathrm{SE}(1)$ are trivial, in our presentation we will use the simpler formalism of classical differential equation theory.\footnote{For instance, in general $\Phi$ should be seen as a vector field on the tangent space of $\R^{2n}$ and \eqref{eq:symmetryaction} should read $\dd\Gamma_g \Phi(t, x,v)=\Phi(t, \Gamma_g (x,v))$, where $\dd \Gamma_g$ is the pushforward of $\Gamma_g$, which in our case corresponds to the identity and is therefore omitted.} 

Clearly, the invariance with respect to the action of the symmetry group $G$ is also inherited by the solutions of \eqref{eq:2n_system}, namely if $(x^*,v^*)$ is a solution of \eqref{eq:2n_system}, so is $\Gamma_g(x^*,v^*)$ for every $g\in \R$. Thus, it is reasonable to study  the reduced problem on the quotient space  $\R^{2n}/G\cong \R^{2n-1}$. 

This problem has a clear interpretation in our setting. Indeed, we notice that given two solutions $(x^*,v^*)$ and $(\hat x,\hat v)$, there exists $g\in G$   such that   $(\hat x,\hat v)=\Gamma_g(x^*,v^*)=(x^*+g\mathds{1}_n ,v^*)$   if and only if they have  the same shape, i.e., $P x^*=P \hat x$, and the same velocities, i.e., $v^*=\hat v$.
Thus, recalling that  $\A\coloneqq P^\top \Ash P\in \R^{n\times n}$, the reduced problem on $\R^{2n}/G\cong \R^{2n-1}$ reads
 \begin{equation}\label{eq:reduced}
	\begin{cases}
		\dot{z}= P v\\
		\mathbb{M}\dot v+F(t,v)+P^\top\Ash  z - \ell(t)=0
	\end{cases}
\end{equation} 

In presence of a symmetry, it is often very natural and convenient to study directly the reduced system. However, some attention should be paid on how properties on the reduced system are related to those on the original one. In our case, we are concerned with periodicity. Given a solution $(x,v)$ of \eqref{eq:2n_system} and its projection $(P x,v)$ solving \eqref{eq:reduced}, it is trivial that if $(x,v)$ is $T$-periodic solution, then also it projection $(P x,v)$ is $T$-periodic. However, the converse does not hold. 
In an equivariant system with a symmetry group $G$, solutions whose projection on the quotient space is periodic are called \emph{relative-periodic}. Accordingly, for our locomotion model \eqref{eqcrawl} we give the following definition.

\begin{defin}\label{def:relativeperiodic}
	A solution $x$ of the rectilinear locomotion model \eqref{eqcrawl} is \emph{relative-periodic} if and only if $(P x, \dot x)$ is a periodic solution of the reduced problem \eqref{eq:reduced}.
\end{defin}
\noindent More precisely, we see that a solution $x(t)$ of  \eqref{eqcrawl} is relative $T$-periodic if and only if
\begin{itemize}
	\item its shape $z(t)=Px(t)$ is $T$-periodic
	\item  its velocity vector $v(t)=\dot x(t)$ is $T$-periodic.
\end{itemize}
It follows that if $\hat x$ is relative $T$-periodic, then $(x^*(T),\dot x^*(T))=\Gamma_g(x^*(0),\dot x^*(0))$ for some $g\in G$, which is usually called \emph{geometric phase} of the relative-periodic solution. In other words, relative-periodic solutions are periodic up to the action of an element of the symmetry group $G$.

Notice that, given a periodic solution of~\eqref{eq:reduced}, it is always possible to reconstruct a corresponding relative-periodic solution of~\eqref{eqcrawl}. 

\begin{lemma}\label{lemma:reconst}
Given a solution $(z^*,v^*)$ of the reduced system~\eqref{eq:reduced}, there exists a  solution $x^*$ of the locomotion model~\eqref{eqcrawl} such that $(Px^*,\dot x^*)=(z^*,v^*)$. Moreover, if $(z^*,v^*)$ is $T$-periodic, then $x^*$ is relative $T$-periodic with geometric phase $g^*$ given by
	\begin{equation}\label{eq:reconstruction}
		g^*= \int_{0}^{T}v_1(t)\dd t 
	\end{equation}
\end{lemma}
\begin{proof}
We define,  by induction 
\begin{align*}
x^*_1(t)\coloneqq\int_0^t v_1(s)\dd s\,, && x^*_i(t)\coloneqq x^*_{i-1}(t)+z^*_{i-1}(t) \quad\text{for $i=2,\dots,n$}
\end{align*}
It is easily verified that this $x^*$ solves \eqref{eqcrawl} and $(Px^*,\dot x^*)=(z^*,v^*)$. 

If, in addition $z^*$ is $T$-periodic, then its derivative $\dot z^*$ has zero mean on $[0,T]$ and consequently we have 
\begin{equation*}
	\int_{0}^{T}v_1(t)\dd t =\int_{0}^{T}v_2(t)\dd t =\cdots= \int_{0}^{T}v_n(t)\dd t 
\end{equation*}
showing that the geometric phase in \eqref{eq:reconstruction} is well-defined.
\end{proof}
\noindent 
Equation \eqref{eq:reconstruction} is usually called \emph{reconstruction equation}.

We see now how well the notion of relative-periodicity grasps the intuitive notion of \emph{gait}: to a prescribed $T$-periodic actuation $(F(t,v),L(t))$ we associate a $T$-periodic shape change and a geometric phase describing the displacement produced by each iteration of the actuation cycle. 
Classical periodicity would be a too strong notion: periodic solutions of \eqref{eqcrawl} have a zero geometric phase, corresponding to an \emph{incompetent} crawler that after each iteration of the gait returns to the initial position.

As  discussed in the introduction, our goal is to show that, given a $T$-periodic actuation, our model asymptotically attains as $t\to+\infty$ a relative-periodic evolution. The  specific behaviour of the crawler we are   aiming to prove  is formalized  in the following definition.

\begin{defin} \label{def:asympt_gait} We say that the locomotion model \eqref{eqcrawl} has an \emph{unique, globally asymptotically stable, relative-periodic behaviour} if both the following conditions hold.
\begin{enumerate}[label=\textup{\roman*)}]
	\item  System \eqref{eqcrawl} has a relative-periodic solution $x^*$ and all the other relative-periodic solutions are $G$-symmetric to $x^*$, namely they are all and only those of the form $\Gamma_g x^*$ for $g\in G$.
	\item Every solution $\xi$ of  \eqref{eqcrawl} satisfies 
	\begin{subequations}
	\begin{align}
	&\lim_{t\to +\infty}P  (x^*(t)-\xi(t))=0 &&\text{(convergence  of the shape)} \label{eq:conv_shape} \\
	&\lim_{t\to +\infty}  \dot{x}^*(t)-\dot{\xi}(t)= 0 &&(\text{convergence  of the velocities}) \label{eq:conv_vel}			
	\end{align}
	\end{subequations}
\end{enumerate}	
\end{defin}

  We emphasize that the key ideas of this section apply to locomotion models and other mechanical systems in general. However, the possibility to express the reduced and reconstruction equations \eqref{eq:reduced},\eqref{eq:reconstruction} as equations respectively on $\R^{2n-1}$ and $\R$, rather than  on  less approachable manifolds, strongly relies on us considering  a simple example of rectilinear locomotion. For instance, if our crawler, possibly with a more complex bi-dimensional body structure, were instead moving on the whole plane, the symmetry group would become $\mathrm{SE}(2)$, which includes not only translations but also rotations. We refer to \cite{FaPaZo} for a nice presentation on the structure of relative-periodicity for some  models of locomotion in a 2D or 3D space, namely with symmetry group $\mathrm{SE}(2)$ or $\mathrm{SE}(3)$. Remarkably, two types of relative-periodic solution emerge, depending on the chosen actuation: (bounded) quasi-periodic solutions and drifting solutions. According to the symmetry group involved, one type  might be prevalent with respect to the other one.

\section{A Massera-type Theorem} \label{sec:massera}
In this section we present a Massera-type result for our model of locomotion. We prove that \emph{for the reduced dynamics} the existence of a bounded orbit implies the existence of a global periodic attractor, meaning that all solutions of the model converge asymptotically to the same shape and the same velocity. An example at the end of the section  suggests that such bounded orbit may not exist. In Section~\ref{sec:bound} we discuss some sufficient conditions for the existence of a bounded solution of the reduced dynamics.

We first state our result for the reduced system \eqref{eq:reduced}.
\begin{theorem}\label{th:Massera_red}
Suppose that   \ref{hyp:A1} \ref{hyp:A2}, \ref{hyp:A3}, \ref{hyp:A4},  \ref{hyp:A5}  hold,  and that  system \eqref{eq:reduced}  admits a solution $(z^\circ,v^\circ)\in W^{1,1}_\mathrm{loc}([0,+\infty);\R^{2n-1})$ which is bounded on $[0+\infty)$, i.e., there exists a constant $C>0$ such that
\begin{align*}
	\norm{z^\circ(t)}+\norm{v^\circ(t)}<C &&\text{for every $t\in[0,+\infty)$.}
\end{align*}
Then system~\eqref{eq:reduced} admits a $T$-periodic solution $(z^*,v^*)\in W^{1,1}_\mathrm{loc}(\R;\R^{2n-1})$. Moreover, every solution $(z,v)$ of ~\eqref{eq:reduced} converges to $(z^*,v^*)$ as $t\to+\infty$, namely
\begin{equation}
\lim_{t\to +\infty}\norm{z^*(t)-z(t)}+\norm{v^*(t)-v(t)}=0
\end{equation}
\end{theorem}

As an immediate consequence of Theorem~\ref{th:Massera_red} we have the following corollary for the locomotion model \eqref{eqcrawl}.
\begin{corol}\label{cor:Massera_loc}
Suppose that  \ref{hyp:A1} \ref{hyp:A2}, \ref{hyp:A3}, \ref{hyp:A4},  \ref{hyp:A5}  hold, and that system  \eqref{eqcrawl} admits a solution bounded in shape and velocity on $[0,+\infty)$, i.e., there exist a solution $x^\circ\in W^{2,1}_\mathrm{loc}([0,+\infty),\R^{2n})$ of \eqref{eqcrawl} and  a constant $C>0$ such that
\begin{align*}
\norm{P x^\circ(t)}+\norm{\dot x^\circ(t)}<C &&\text{for every $t\in[0,+\infty)$.}
\end{align*}
Then Equation~\eqref{eqcrawl} admits a unique, globally asymptotically stable, relative $T$-periodic behaviour in the sense of Definition~\ref{def:asympt_gait}.
\end{corol}

In order to prove Theorem~\ref{th:Massera_red}, we first show  a contraction-type result for the reduced dynamics~\eqref{eq:reduced}. 

\begin{theorem}\label{th:contrazione}
Assume that  \ref{hyp:A1} \ref{hyp:A2}, \ref{hyp:A3}, \ref{hyp:A4},  \ref{hyp:A5}  hold. Let $(z,v)$ and $(\zeta,w)$   be two solutions of \eqref{eq:reduced} and assume that $v$ is bounded, namely that there exists a constant $C_v>0$ such that $\abs{v(t)}<C_v$ for every $t\geq 0$. Then
\begin{subequations}
\begin{align}
&\lim_{t\to +\infty}z(t)-\zeta(t)=0 &&\text{(convergence  of the shape)}\label{eq:convform}\\
&\lim_{t\to +\infty}  v(t)-w(t)= 0 &&(\text{convergence  of the velocities}) \label{eq:convelo}
\end{align}
\end{subequations}
\end{theorem}
\begin{proof}

We define the function 
\begin{equation}\label{eq:defE}
	E(z,\zeta,v,w)\coloneqq\frac{\norm{v-w}^2_{\mathbb{M}}}{2}+\frac{\norm{z-\zeta}^2_{\Ash}}{2}
\end{equation}  
so that
\begin{equation}\label{eq:dotE}
	\dot E=\langle \mathbb{M}(\dot v-\dot w), v-w\rangle+ \langle \Ash(z-\zeta), \dot z-\dot \zeta\rangle
\end{equation}
for a.e. $t\in [0,+\infty)$. Using the second equation in \eqref{eq:reduced} for each solution we get
\begin{equation}\label{eq:deriv1}
	\langle \mathbb{M}(\dot v-\dot w), v- w\rangle=-\langle F(t,v)-F(t, w), v- w\rangle-\langle P^\top \Ash (z-\zeta) , v- w\rangle
\end{equation}
for a.e. $t\in [0,+\infty)$.
Combining \eqref{eq:dotE} and \eqref{eq:deriv1}  we obtain 
\begin{equation}\label{eq:dotEb}
	\dot E=-\langle F(t,v)-F(t, w), v- w\rangle\leq -\hat \mu\norm{v- w}^2\leq -\hat \mu C_{\mathbb{M}}\norm{v- w}_{\mathbb{M}}^2 
\end{equation}
 a.e. $t\in [0,+\infty)$,  where the first inequality follows from \ref{hyp:A4} and the second from the equivalence  of norms in $\R^n$, for a suitable constant $C_{\mathbb{M}}>0$.

Let us now define the function
\begin{equation*}
\psi(z,\zeta,v, w)\coloneqq \hat \mu C_{\mathbb{M}} \norm{v- w}^2_{\mathbb{M}} \qquad \text{a.e. $t\in[0, +\infty)$.}
\end{equation*}
 so that the estimate in \eqref{eq:dotEb}  reads $\dot E\leq -\psi$,   a.e. $t\in [0,+\infty)$.
Taking into account \eqref{eq:deriv1} and \ref{hyp:A4} we get  
\begin{equation*}
	\dot\psi\leq - 2\hat \mu C_{\mathbb{M}}\langle P^\top\Ash(z-\zeta), v- w\rangle \qquad \text{a.e. $t\in[0, +\infty)$.}
\end{equation*}
  By the estimate in~\eqref{eq:dotEb} we deduce that  $E(t)$ is decreasing and therefore  $0\leq E(t)\leq E(0)$ for every $t\geq 0$. By \eqref{eq:defE} and the equivalence  of norms in $\R^k$ it follows that 
 $\norm{v(t)- w(t)}$ and $\norm{\zeta(t)-z(t)}$  are bounded on $[0,+\infty)$;  hence, there exists $B\in \R$  such that 
\begin{equation*}
	\dot\psi\leq B \qquad \text{a.e. $t\in[0, +\infty)$.}
\end{equation*}
 We can now apply  Theorem~4.5 in \cite[pag.~287]{RouHaLa}\footnote{ We note that  this theorem is a corollary of Lemma~4.3 in \cite[pag.~286]{RouHaLa} and that both results are stated  assuming suitable inequalities  on the    Dini derivatives $D^+$ of the functions involved. However,    one can  adapt in a quite direct manner the proof of the  Lemma, and hence of the theorem,  to our  Carath\'eodory   setting by replacing the Dini derivatives $D^+$, defined everywhere on $[0, +\infty)$, with the derivatives of the absolutely continuous functions involved, defined almost everywhere on $[0,+\infty)$.} with $\Omega=S=\R^{4n-2},\, a(r)=r,\, V=E,  I=[0, +\infty)$  to the  following  first order system,  obtained by considering two copies of system \eqref{eq:reduced},
\begin{equation}\label{eq:doublereduced}
	\begin{cases}
		\dot{z}= P v\\
\mathbb{M}\dot v+F(t,v)+P^\top\Ash  z - \ell(t)=0 \\
		\dot{\zeta}= P w\\
\mathbb{M}\dot w+F(t,w)+P^\top\Ash  \zeta - \ell(t)=0
	\end{cases}
\end{equation}
and conclude that 
\begin{equation*}
\lim_{t\to +\infty}\psi(t)=0.
\end{equation*}
This proves  \eqref{eq:convelo}.

For $0<\alpha<T$, let us define 
\begin{equation*}
	\Theta(t)\coloneqq \int_t^{t+\alpha}\left[z(s)-\zeta(s)\right]\dd s   \,.
\end{equation*}
We  prove first a weaker form of \eqref{eq:convform}, namely that, for every $\alpha>0$ and every $i=1,\dots,n-1$
\begin{equation}\label{eq:shape_comp_conv}
	\lim_{t\to+\infty}\Theta_i(t)=\lim_{t\to+\infty}\int_t^{t+\alpha}\left[z_i(s)-\zeta_i(s)\right]\dd s =0  \,.
\end{equation}
Subtracting the fourth equation from the second one in \eqref{eq:doublereduced} and integrating the result in the time interval $[t,t+\alpha]$ we get for  the $i$-th component
\begin{equation}\label{eq:conv_integral}
	0=\int_{t}^{t+\alpha}\left[ m_i\bigl(\dot v_i(s)-\dot w_i(s)\bigr)+F_i(s,v_i)-F_i(s,w_i) +e_i^\top P^\top\Ash \bigl(z(s)-\zeta(s)\bigr)\right] \dd s
\end{equation}
where $e_i$ is the $i$-th   vector of the canonical basis of $\R^n.$ 

By the already proven \eqref{eq:convelo},  we deduce that for every $\epsilon>0$ there exist a $\tau_\epsilon>0$ such that, for every $s\geq \tau_\epsilon$, it holds $\abs{v(s)-w(s)}<\epsilon$.

Let us define the compact set $K=[-C_v-1,C_v+1]^n\subseteq \R^n$. Hence, for every $\epsilon\in(0,1)$ and for every time $t>\tau_\epsilon$ we have $v(t),w(t)\in K$. By \ref{hyp:A2}, for every $\epsilon\in(0,1)$ and every $t\geq\tau_\epsilon$ 
\begin{equation}\label{eq:F_estimate}
\int_t^{t+\alpha}\abs{F_i(s,v_i(s))-F_i(s,w_i(s))}\dd s \leq \alpha \Lambda_K \epsilon 
\end{equation}

Thus, for every $\epsilon\in(0,1)$,  every $t\geq\tau_\epsilon$ and every index $i=1,\dots,n$, it holds
\begin{align*}
	\abs{e_i^\top P^\top \Ash\Theta(t)}&=\abs{e_i^\top P^\top \Ash \int_t^{t+\alpha} [z(s)-\zeta(s)]\dd s}\\ &\leq\Bigl[m_i\abs{v_i(t+\alpha)-w_i(t+\alpha)}+m_i\abs{v_i(t)-w_i(t)}  +\int_t^{t+\alpha}\abs{F_i(s,v_i)-F_i(s,w_i)}\dd s \Bigr]\\
	&\leq \left(2m_i+\alpha\Lambda_K\right)\epsilon
\end{align*}

It follows that, for every $i$,  we have $\abs{e_i^\top P^\top \Ash\Theta(t)}\to 0$ as $t\to+\infty$ and, consequently, $\norm{P^\top\Ash\Theta(t)}\to 0$. We observe that the $n\times(n-1)$ matrix $P^\top\Ash$ has maximum rank, i.e., it has rank $n-1$: indeed,  $\Ash$ is an invertible $n-1$ square matrix, while $P^\top$ is an $n\times(n-1)$ matrix with maximum rank.
Thus, it follows that $\norm{\Theta(t)}\to 0$  as $t\to+\infty$ and, accordingly, \eqref{eq:shape_comp_conv}.

It remains to show that the convergence \eqref{eq:shape_comp_conv} implies the stronger property \eqref{eq:convform}, i.e,~$\lim_{t\to +\infty} z(t)-\zeta(t)=0$. To do so, we proceed by contradiction. Suppose that there exists $i\in\{1,\dots,n-1\}$ and a sequence $t_h\to+\infty$ such that $\abs{z_i(t_h)-\zeta_i(t_h)}> \delta$ for some $\delta>0$.  By the continuity of $\dot x$ and $ \dot \xi$,  and by \eqref{eq:convelo}, there exists a constant $c>0$ such that  $\abs{\dot z_i(t)-\dot \zeta_i(t)}<c$ for every $t\in [0, +\infty)$.
 Let us now choose $\alpha$ such that $0<\alpha<\frac{\delta}{c}$. Thus
\begin{equation}\label{eq:integral_est}
\int_{t_h}^{t_h+\alpha}\abs{z_i(s)-\zeta_i(s)}\dd s>\alpha\left( \abs{z_i(t_h)-\zeta_i(t_h)}-\frac{c\alpha}{2}\right)>\alpha\left(\delta-\frac{\delta}{2}\right)=\frac{\alpha\delta}{2}>0
\end{equation}
Since \eqref{eq:integral_est} holds for every $h$, we get a contradiction with \eqref{eq:shape_comp_conv}, proving \eqref{eq:convform}.

\end{proof}

We notice that also the following variation of Theorem~\ref{th:contrazione} holds, where the bound on the velocity of one solution is replaced by the global Lipschitz continuity of the friction~$F$.

\begin{corol} \label{cor:contrazione_lip} Assume that   \ref{hyp:A1} \ref{hyp:A2}, \ref{hyp:A3}, \ref{hyp:A4},  \ref{hyp:A5}  hold. Let $(z,v)$ and $(\xi,w)$   be two solutions of \eqref{eq:reduced} and assume that each $F_i(t,v)$ is globally  Lipschitz continuous in $v$ uniformly in $t\in[0,T]$ for the same Lipschitz constant $C_F>0$. Then \eqref{eq:conv_shape}--\eqref{eq:conv_vel} hold.
\end{corol}
\begin{proof}
The proof closely follows that of  Theorem~\ref{th:contrazione}, with a minor adjustment. Specifically, we notice that the bound $\abs{v(t)}<C_v$ of  Theorem~\ref{th:contrazione} is required only to to construct the compact set $K$ where local Lipschitz continuity from \ref{hyp:A2} can be employed to obtain the estimate \eqref{eq:F_estimate}.  However, if global Lipschitz continuity is assumed instead, then \eqref{eq:F_estimate}  holds directly  with $C_F$ replacing$\Lambda_K$.
\end{proof}

 We are now ready to prove our main result for the reduced system \eqref{eq:reduced}.
 
 \begin{proof}[Proof of Theorem~\ref{th:Massera_red}]
Let  $x^\circ(t)$ and $C>0$ be  as in the statement of the theorem. Define the compact set 
\begin{equation*}
{\mathcal{K}}\coloneqq \{(z,v)\in \R^{2n-1}\,:\, \norm{z}+\norm{v}\leq 2C\}.
\end{equation*}
Fix  $(z_0,v_0)\in \R^{2n-1}$  and let $(z(t), v(t))$ be the solution of \eqref{eq:reduced} such that $(z(0),v(0))=(z_0, v_0)$. By Theorem \eqref{th:contrazione}, we have that there exists $t_0=t_0(z_0, v_0)>0$ such that 
$(z(t),v(t))\in {\mathcal{K}}$  for any $t\geq t_0(z_0, v_0)$. 

Then, system \eqref{eq:reduced}  is point dissipative,  and by   \cite[Corollary 2.1]{Pliss}  we get  that system ~\eqref{eq:reduced} admits a $T$-periodic solution $(z^*,v^*)$.   In our setting such solution belongs  to $W^{1,1}_\mathrm{loc}(\R;\R^{2n-1})$. Moreover, applying again Theorem \ref{th:contrazione},  we see that  every other solution $(z,v)$ of ~\eqref{eq:reduced} converges to $(z^*,v^*)$ as $t\to+\infty$, namely
\begin{equation}
\lim_{t\to +\infty}\norm{z^*(t)-z(t)}+\norm{v^*(t)-v(t)}=0.
\end{equation}
Our proof is concluded.
 \end{proof}

 We end this section with an  example   which shows  that  the shape and the velocities of our locomotor may  not be bounded on $[0,+\infty)$. The system considered in the example is such that  the friction forces are globally bounded, so condition  \ref{hyp:A4} does not hold; notice, however, that this example covers smooth approximations of dry friction.

\begin{example} \label{ex:resonance}

\begin{figure}[th]
		\centering
		\subcaptionbox{Shape $z(t)$.\label{subfig:res_shape}}
		{	\includegraphics[scale=0.2]{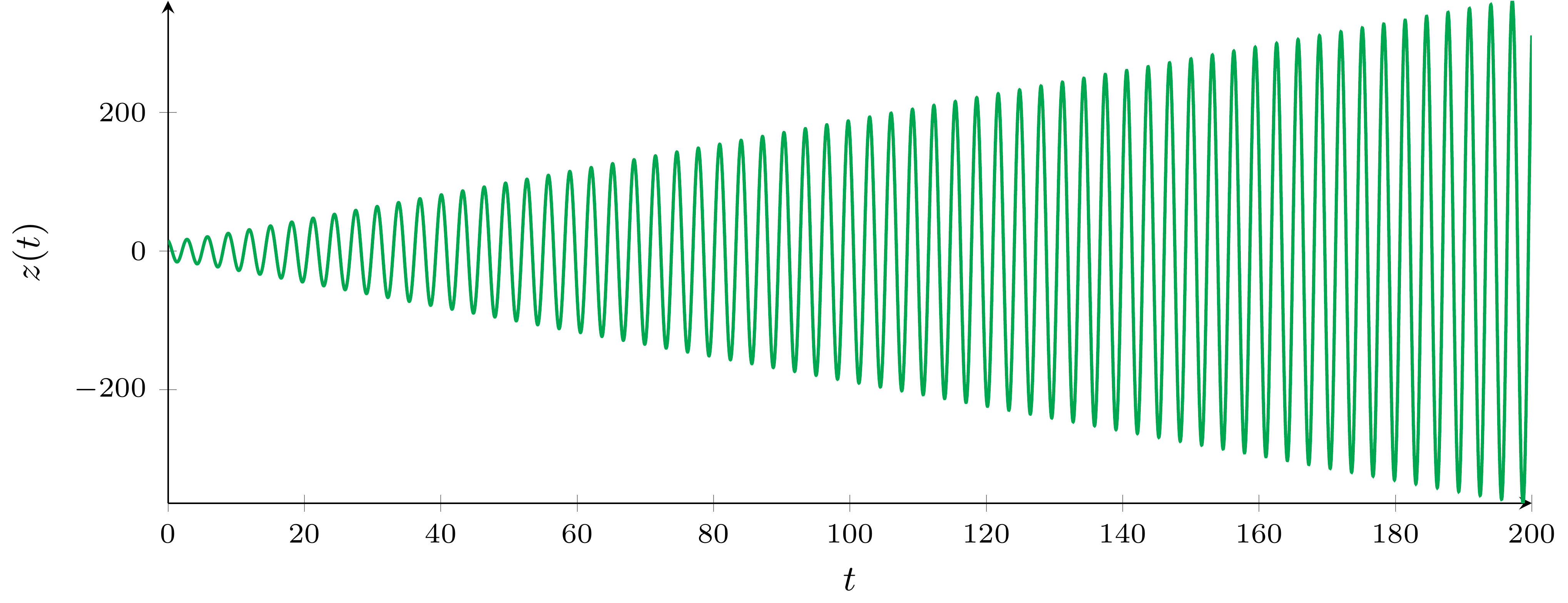}
			}\\[1cm]
		\subcaptionbox{Velocity $\dot{\bar x}(t)=\frac{1}{2}(\dot x_1(t)+\dot x_2(t))$ of the barycenter.\label{subfig:res_vel}}
		{  \includegraphics[scale=0.2]{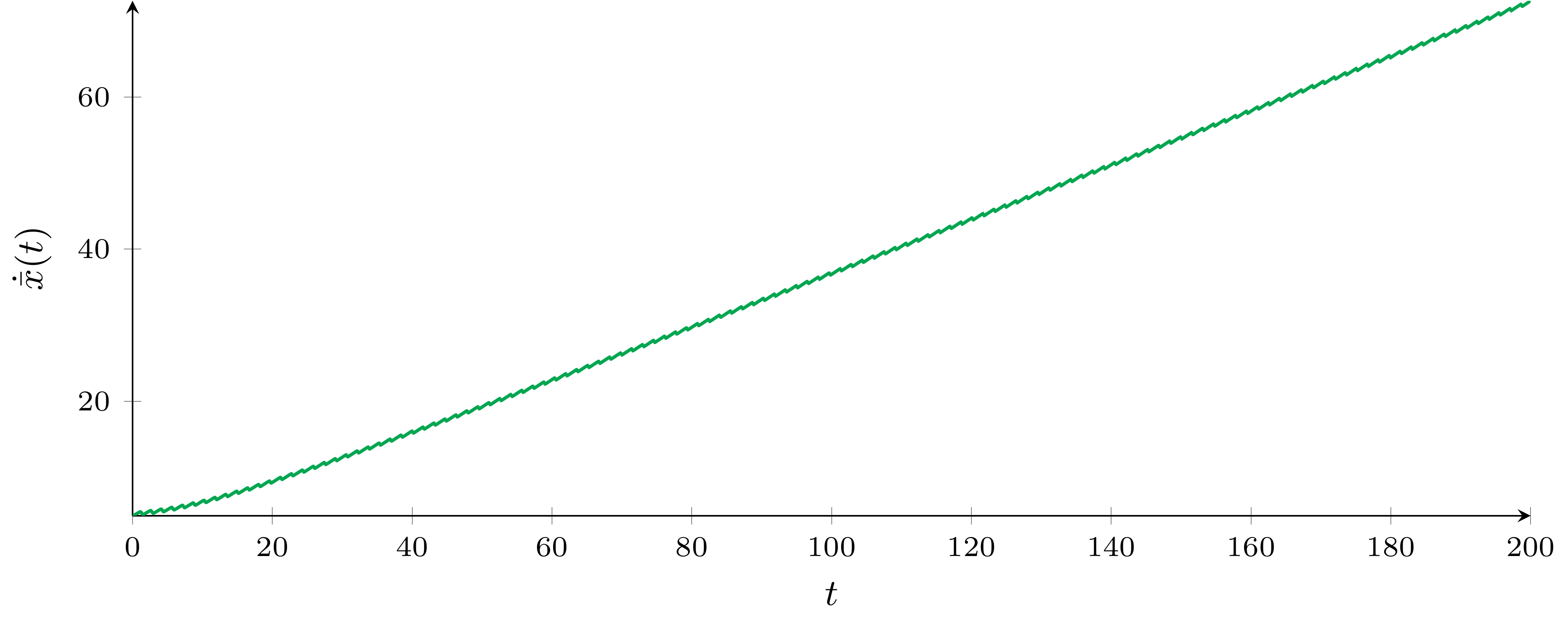}}
		\caption{Plots of  the shape and velocity of the barycenter of the solution of System \eqref{eq:example2} in Example~\ref{ex:resonance}  with parameters $k=2$, $A=3$,  $F_1(t,v)=F_2(t,v)=\frac{1}{4}\arctan(v)(5-\arctan(v)$ and  initial conditions  $(x_1(0),\dot{x}_1(0),  x_2(0),\dot{x}_2(0))=(0,10,15,0)$.}\label{fig:resonance}
\end{figure}
Consider    equation \eqref{eqcrawl} with  $n=2$,  $m_1=m_2=1$, $\Ash= k$,   $L_1(t)=A\cos(\sqrt{2k} t),\,\,A>0$, and $F_i(t, v)$   satisfying   \ref{hyp:A1},  \ref{hyp:A2}, \ref{hyp:A3},  \ref{hyp:Dfric}    and  such that 
\begin{equation}\label{eq:bound}
 \abs{F_1(t,v_1)}+\abs{F_2(t,v_2)}\leq\alpha \in (0, kA),\quad (t,v_1,v_2)\in\R\times \R\times\R
\end{equation}
The resulting system is
\begin{equation}\label{eq:example2}
	\begin{cases}
		\ddot{x}_1=-F_1(t,\dot{x}_1)+k(x_2-x_1)-kA\cos (\sqrt{2k}t)\\
\ddot{x}_2=-F_2(t,\dot{x}_2)+k(x_1-x_2)+kA\cos(\sqrt{2k} t) 
	\end{cases}
\end{equation}
Let $(x_1(t), x_2(t)),\,\,t\in [0,+\infty)$   be any solution of system \eqref{eq:example2}
and let 
\begin{equation*}
p(t)\coloneqq -F_2(t,\dot{x}_2(t))+F_1(t,\dot{x}_1(t)).
\end{equation*}
The corresponding shape $z(t)=x_2(t)-x_1(t)$  is a solution of 
\begin{equation}\label{eq:res_osc}
    \ddot{w}+2kw=p(t)+2kA\cos(\sqrt{2k} t) 
\end{equation}
All solutions of \eqref{eq:res_osc}, and hence also $z(t)$,   
satisfy 
\begin{subequations}
\begin{align}
\limsup_{t\to +\infty} w(t)=+\infty\label{eq:unboundedplus}\\
\liminf_{t\to +\infty}  w(t)=-\infty\label{eq:unboundedminus}
\end{align}
\end{subequations}
In fact, any solution of \eqref{eq:res_osc} is of the form 
\begin{equation*}
w(t)=w_h(t)+\int_0^t \sin(\sqrt{2k}(t-s))[p(s)+2kA\cos(\sqrt{2k} s)]\,ds,
\end{equation*}
where   $w_h(t)=C_1\cos(\sqrt{2k} t)+C_2\sin(\sqrt{2k} t),\,\, C_1, C_2\in \R$. Since  
\begin{equation*}
\int_0^t \sin(\sqrt{2k}(t-s))\cos(\sqrt{2k} t)\,ds=\frac{t}{2}\sin(\sqrt{2k} t)
\end{equation*}
from \eqref{eq:bound} we have
\begin{equation*}
(-\alpha+ kA\sin(\sqrt{2k}t))t+w_h(t)\leq w(t)\leq (\alpha+ kA\sin(\sqrt{2k}t))t+w_h(t).
\end{equation*}
Then, if we define the sequence $t_h\coloneqq \frac{1}{\sqrt{2k}}(2h+1/2)\pi,\,\,h\in \mathbb{N}$, by the first   inequality   we get  immediately  
\begin{equation*}
\lim_{h\to +\infty}w(t_h)=+\infty
\end{equation*} 
and \eqref{eq:unboundedplus} follows.   In a similar manner, using the second  inequality with   $\tau_h\coloneqq t_h+\frac{\pi}{\sqrt{2k}},\,\,h\in \mathbb{N}$,  one gets \eqref{eq:unboundedminus}. From  this unbounded oscillatory character of solutions it follows also that the velocity  $v(t)=(\dot{x}_1(t), \dot{x}_2(t))$  is unbounded  on $[0,+\infty)$. 

An illustration of the behaviour of solutions of system \eqref{eq:example2} when \eqref{eq:bound}  holds is given in Figure~\ref{fig:resonance}, where we have simulated the system  choosing  $k=2$, $A=3$,  $F_1(t,v)=F_2(t,v)=\frac{1}{4}\arctan(v)(5-\arctan(v))$. In this case we can take  $\alpha=\frac{5}{4}\pi+\frac{\pi^2}{8}\in (0,6)$.
\end{example}

We notice that Example~\ref{ex:resonance} relies on an  internal resonance of the system.
If we eliminate the resonance by changing the period of the input, considering  for example $L_1(t)=\cos(t)$,  the  numerical simulations 
differ significantly and  suggest the possible existence of a global periodic attractor.

\begin{remark} \label{rem:displacement}
The Reader might be wondering why in Example~\ref{ex:resonance} we considered a sign-changing shape $z(t)$, which might seem to lead to a change in the order of the blocks.
Although mechanisms allowing such an inversion might be produced, a better explanation is available. While we refer for simplicity to $x_i$ as the position of the $i$-th block, as already noticed it should properly be considered as the \emph{displacement with respect to a} (constant) \emph{reference configuration} $\xi_i$. Hence, the \emph{actual position} of the block is $\xi_i+x_i$ and therefore the actual signed-distance between consecutive blocks is $\xi_{i+1}+x_{i+1}-\xi_i-x_i=z_i+(\xi_{i+1}-\xi_i)$, so that a negative $z_i$ does not necessarily mean the inversion of the blocks. Notice that any choice of the reference configuration leads to the very same dynamics \eqref{eqcrawl}, up to embedding the constant $\xi_{i+1}-\xi_i$ in $L_i(t)$, so our setting is not restrictive.

Clearly, these considerations allow the shape $z$ to fluctuate around zero, but do not exclude block-inversion for sufficiently large shape changes. However, in actual devices, we expect other phenomena, such as nonlinear  elastic forces, plasticity or damage in the links, to occur for deformations smaller than the ones required for block-inversion, so that the latter is not a limiting factor for the range of validity of the model.
\end{remark}

%%%%%%%%%%%%%%%%%%%%%%%%%%%%%%%%%%%%%%%%%%%%%%%%%%%%%%%%%%%%%%%%%%%%   
\section{Existence of a limit relative-periodic behaviour} \label{sec:bound}

Theorem~\ref{th:Massera_red}, like  all Massera-type theorems in general,  solves only half of the problem, leaving  open the question of whether a solution of \eqref{eq:reduced} bounded  on $[0,+\infty)$ exists. As we saw in Example~\ref{ex:resonance}, this may not be the case. 

We will prove the existence of such a bounded solution, thus establishing the existence of an unique limit relative-periodic behaviour, in two special situations.

The first case we consider  is that of time-dependent viscous forces, which we will analyse in detail in Section~\ref{sec:viscous}.  Our main result, proved in Section~\ref{sec:viscous}, is the following.
\begin{theorem}
\label{th:attractor_viscous}
	Let us assume that, in addition to  \ref{hyp:A4} and  \ref{hyp:A5} , friction forces  are of the form 
 \begin{equation}\label{eq:viscous}
F(t,v)=\diag(\mu_1(t),\cdots,\mu_n(t))\,v\coloneqq \V(t) \,v 
\end{equation}
where  $\mu_i\in L^\infty(\R, [0,+\infty))$  are   $T$-periodic functions. 
Then Equation~\eqref{eqcrawl} admits a unique, globally asymptotically stable, relative-periodic behaviour in the sense of Definition~\ref{def:asympt_gait} and the convergences \eqref{eq:conv_shape}--\eqref{eq:conv_vel} are exponential.
\end{theorem}
 Notice that \ref{hyp:A1} \ref{hyp:A2}, \ref{hyp:A3} are automatically satisfied by \eqref{eq:viscous}.
Moreover, under \eqref{eq:viscous}, condition \ref{hyp:A4} is equivalent to assume the existence of $\hat \mu>0$ such that $\mu_i(t)\geq \hat \mu$  on  $[0,T]$.
We remark that, for viscous friction, the dynamics of the system becomes linear and an explicit form of the solutions is available, which allows us to recover additional information on the system. In particular, in Proposition~\ref{prop:incomp}, we show that if the friction matrix $\V$ is constant in time, then the geometric phase of the corresponding relative-periodic solution is zero. In other words, in its relative-periodic limit the crawler becomes incompetent, with each iteration of the gait returning the crawler to its initial state, namely $x(t+T)=x(t)$. Notice that a time-dependent viscous friction, although not fitting most concrete examples, is a key feature in the locomotion of microscale hydrogel crawlers \cite{KropacekAl,RehorAl21}. 

\medskip

Outside of the (linear) viscous case, the existence of a bounded solution is a much more challenging issue. We prove it for a special case, in which we assume that all the blocks have the same mass $\bar{m}$ and the same friction law $F_i=f$, and that the body is sufficiently stiff. More precisely, we make the following additional assumption:
\begin{enumerate}[label=\textup{(A\arabic*)},start=6]
	\item $m_1=m_2=\cdots m_n\eqqcolon \bar{m}$ and $F_1=F_2=\dots=F_n\eqqcolon f$, where $f(t,u)$ is $C_f$-Lipschitz continuous in $u$ uniformly in $t$, and such that the Lipschitz constant $C_f$ satisfies 
\begin{equation*}	
	C_f<4k(1-\cos\frac{\pi}{n-1})
\end{equation*}
	where $k>0$ denotes the minimum eigenvalue of $\Ash$. \label{hyp:stiffbody}
\end{enumerate}
We have the following result.
\begin{theorem}\label{th:attractor_stiff}
Let us assume that, in addition to the assumptions of Section~\ref{sec:model}, system~\eqref{eqcrawl} satisfies also \ref{hyp:stiffbody}. 	Then Equation~\eqref{eqcrawl} admits a unique, globally asymptotically stable, relative-periodic behaviour in the sense of Definition~\ref{def:asympt_gait}.
\end{theorem}

When dealing with systems that include  a dissipation mechanism,  the existence of \emph{one} solution which is bounded forward in time, as required by Massera-type theorems,  can be established by demonstrating the system's point dissipativity, i.e. the existence of a compact global attractor. As  previously  mentioned, in our setting this approach  is not restrictive, since by Theorem~\ref{th:contrazione} the existence  of such an attractor follows from the  existence of one  bounded orbit.

The following lemma gives a sufficient condition for the point dissipativity of system system~\eqref{eqcrawl}.
\begin{lemma} \label{lemma:stiff_dissip} Let us assume that system~\eqref{eqcrawl} satisfies all the assumptions of Section~\ref{sec:model} and \ref{hyp:stiffbody}. Then  there exists a compact set $K\subset \R^{2n-1}$ such that for every initial condition 
	\begin{equation}\label{eq:initialvalue}
		(z(0),v(0))=(z_0,v_0)\in \R^{2n-1}
	\end{equation}	
	there exists a time $t_0=t_0(z_0,v_0)$ such that the corresponding solution of \eqref{eq:reduced} satisfies $(z(t),v(t))\in K$ for every $t\geq t_0$.
\end{lemma}
The proof of Lemma~\ref{lemma:stiff_dissip} is given in Section~\ref{sec:stiff}. 

\begin{proof}[Proof of Theorem~\ref{th:attractor_stiff}]
	By Lemma~\ref{lemma:stiff_dissip} we deduce that each solution of \eqref{eq:reduced} is bounded forward in time, hence the Theorem follows straightforwardly by Corollary~\ref{cor:Massera_loc}.
\end{proof}

%%%%%%%%%%%%%%%%%%%%%%%%%%%%%%%%%%%%%%%%%%%%%%%%%%%%%%%%%%%%%%%%%%%%%

\section{Time-dependent viscous friction} \label{sec:viscous}

In this section  we consider the case of viscous friction as defined in \eqref{eq:viscous}; we first demonstrate Theorem~\ref{th:attractor_viscous} and then discuss some further properties of this special case.
 
\begin{proof}[Proof of Theorem~\ref{th:attractor_viscous}]

We begin by noticing that, in this setting,  if we let $\eta=(z,v)$,   system   \eqref{eq:reduced} is the linear system    
\begin{equation}\label{eq:linear} 
\dot{\eta}=\widetilde A(t)\eta+ \widetilde B(t), \qquad \text{a.e. $t\in[0, +\infty)$}
\end{equation}
where
\begin{align*}
\widetilde A(t)=
	\begin{bmatrix} 
	 0_{(n-1)\times (n-1)} & P\\[2mm]
	 -\mathbb{M}^{-1} P^\top \Ash & -\mathbb{M}^{-1}\V(t)
	\end{bmatrix} &&  
\widetilde B(t)= \begin{bmatrix} 
	0_{(n-1)\times 1}\\[2mm]
	\mathbb{M}^{-1}\ell(t)
	\end{bmatrix}
\end{align*}
We therefore have to prove the existence of a unique $T$-periodic solution $\eta_*=(z_*, v_*)$ of  \eqref{eq:linear} and that such $\eta^*$ is a global attractor of the dynamics.

Let $\eta_a=(z_a,v_a)$ and $\eta_b=(z_b,v_b)$ be two solutions of \eqref{eq:linear} and set $d=\eta_a-\eta_b$. Thus, $d$ is a solution of the homogeneous $T$-periodic system
\begin{equation} \label{eq:linear_homogeneous}
	\dot d=\widetilde{A}(t)d\qquad \text{a.e. $t\in[0, +\infty)$}.
\end{equation}
 By Corollary~\ref{cor:contrazione_lip} we know that $d(t)\to 0$ as $t\to +\infty$; therefore by \cite[Theorem 7.1]{Hale80} all characteristic multipliers of \eqref{eq:linear_homogeneous}  have modulus less than one   (and correspondingly all characteristic exponents have negative real part) and, moreover, there exists $K>0$, $\alpha>0$ such that 
 \begin{equation*}
  \norm{d(t)}\leq \norm{d(0)}Ke^{-\alpha t}\,,\qquad t\geq 0.    
 \end{equation*}
 
 Thus, all the solutions of system \eqref{eq:linear} are  bounded  on $[0,+\infty)$.
 The Theorem follows by Corollary~\ref{cor:Massera_loc}.
\end{proof}

In the linear setting, we   can say something more specific about the asymptotic dynamics.
Let us first introduce the center of mass of the system as $\bar x=\frac{1}{M}\sum_{i=1}^{n}m_ix_i$, where $M=\sum_{i=1}^n m_i$ is the total mass of the system.

It is well known that $M\ddot{\bar x}=F^\mathrm{ext}$, where $F^\mathrm{ext}$ is the sum of the external forces acting on the system. In our case, since elastic forces are all internal, $F^\mathrm{ext}$ is the sum of the friction forces acting on each block,  and  therefore
\begin{equation} \label{eq:1card}
	\ddot{\bar x}(t)=-\frac{1}{M}\sum_{i=1}^n \mu_i(t)\dot x_i(t) \qquad \text{a.e. $t\in[0, +\infty)$}.
\end{equation}
Let us write $\alpha_i(t)=x_i(t)-\bar x(t)$ and notice that 
\begin{equation}\label{eq:baric_decomp}
\alpha_i=\frac{1}{M}\sum_{j=1}^n m_j(x_i-x_j) 
\end{equation}
Since, for every $i>j$, the difference $x_i-x_j$ can be expressed as the sum of the $z_h$ associated to all the links between the $i$-th and the $j$-th blocks, it follows that each function $\alpha_i(t)$ is a linear function of the shape $z(t)$; in particular, there exists a (constant) matrix $Q\in\R^{n\times(n-1) }$ such that $\alpha(t)\coloneqq (\alpha_1(t),\dots,\alpha_n(t))^\top=Q z(t)$.

Writing $u=\dot{\bar x}$,   can reformulate \eqref{eq:1card} as
\begin{equation}
	\dot u(t)=-\frac{1}{M}\sum_{i=1}^n \mu_i(t)[u +\dot \alpha_i(t)]=
	-r(t)u(t)+b(t)  \qquad \text{a.e. $t\in[0, +\infty)$}
\end{equation}
where, recalling that $\mathds{1}_n=(1,1,\dots,1)^\top\in \R^{n\times 1}$, we set
\begin{align*}
r(t)=\frac{1}{M}\sum_{i=1}^n \mu_i(t)>\frac{n\hat\mu}{M} && b(t)=-\frac{1}{M}\mathds{1}_n^\top\V(t)Q\dot z(t) \,.
\end{align*}

The solutions of this equation have the form
\begin{equation}\label{eq:linearsol}
	u(t)=e^{-\int_{0}^{t}r(s)\dd s}u(0)+\int_{0}^{t}e^{-\int_{s}^{t}r(\tau)\dd \tau}b(s)\dd s
\end{equation}

We now focus on the special case of a constant friction matrix $\V$, for which it holds $r(t)=\frac{1}{M}\tr \V\eqqcolon r>\frac{n\hat\mu}{M}$ .

\begin{prop} \label{prop:incomp}
If $\V$ does not depend on time, then the relative-periodic solutions $x^*$ of the dynamics \eqref{eqcrawl} satisfy $\int_0^T\dot{\bar x}^*(t)\dd t=0$. In particular, by Theorem~\ref{th:attractor_viscous} it follows that any generic solution of \eqref{eqcrawl} with center of mass $\bar x(t)$ satisfies
\begin{equation*}
\lim_{t\to+\infty} \frac{\bar x(t)-\bar x(0)}{t}=0
\end{equation*}
that is, the crawler is asymptotically incompetent. 
\end{prop}
\begin{proof}
By Theorem~\ref{th:attractor_viscous}, all relative-periodic solutions have the same shape $z^*(t)$ and the same velocities $v^*(t)$, so in particular they have the same velocity of the center of mass $\dot{\bar x}^*(t)\eqqcolon u^*(t)$ and both $u^*(t)$  and $z^*(t)$ are  $T$-periodic. Then, writing $\alpha^*\coloneqq Qz^*$, $u^*(t)$ satisfies  
the differential equation
\begin{equation}
	\dot v(t)=-\frac{1}{M}\sum_{i=1}^n \mu_i[v +\dot \alpha_i^*(t)] \qquad \text{a.e. $t\in[0, +\infty)$}
\end{equation}
which can be thought of as corresponding to  a locomotor with periodically prescribed shape subject to constant viscous friction. For such locomotors it was proved in  \cite[Example 3.4]{GiMaRe}  that there exists a unique globally attracting $T$-periodic 
solution $v^*$ which has zero average on the period. By uniqueness $v^*=u^*$   and we conclude that $\int_{0}^{T}u^*(t)\dd t=0 $. This proves the first part of the statement.

As a consequence, for a generic solution, the sequence of differences $\bar x(jT)-\bar x ((j-1)T)$ converges to zero as $j\to+\infty$. Since the arithmetic mean of the first $j$ terms of a converging sequence converges to the limit of the sequence  as $j\to+\infty$, the second part of the statement easily follows.
\end{proof}

The incompetence of a discrete model of crawler with constant viscous friction was already observed in \cite{DeSTat12} in the quasistatic case with prescribed shape, and later in \cite{GiMaRe} in the dynamic case with prescribed shape. Different improvements can be applied to the viscous case in order to attain true locomotion capability, such as a continuous body \cite{DeSTat12} or considering time-dependent viscous forces \cite[Example~3.5]{GiMaRe}.
 In our framework, where shape is not prescribed and the elastic energy is considered,  we provide the following example of an effective discrete model of crawler employing time-dependent viscous friction to achieve a nonzero limit average velocity. 

\begin{figure}[t]
	\centering
	\subcaptionbox{Shape $z(t).$\label{subfig:shape}}
	{\includegraphics[scale=0.2]{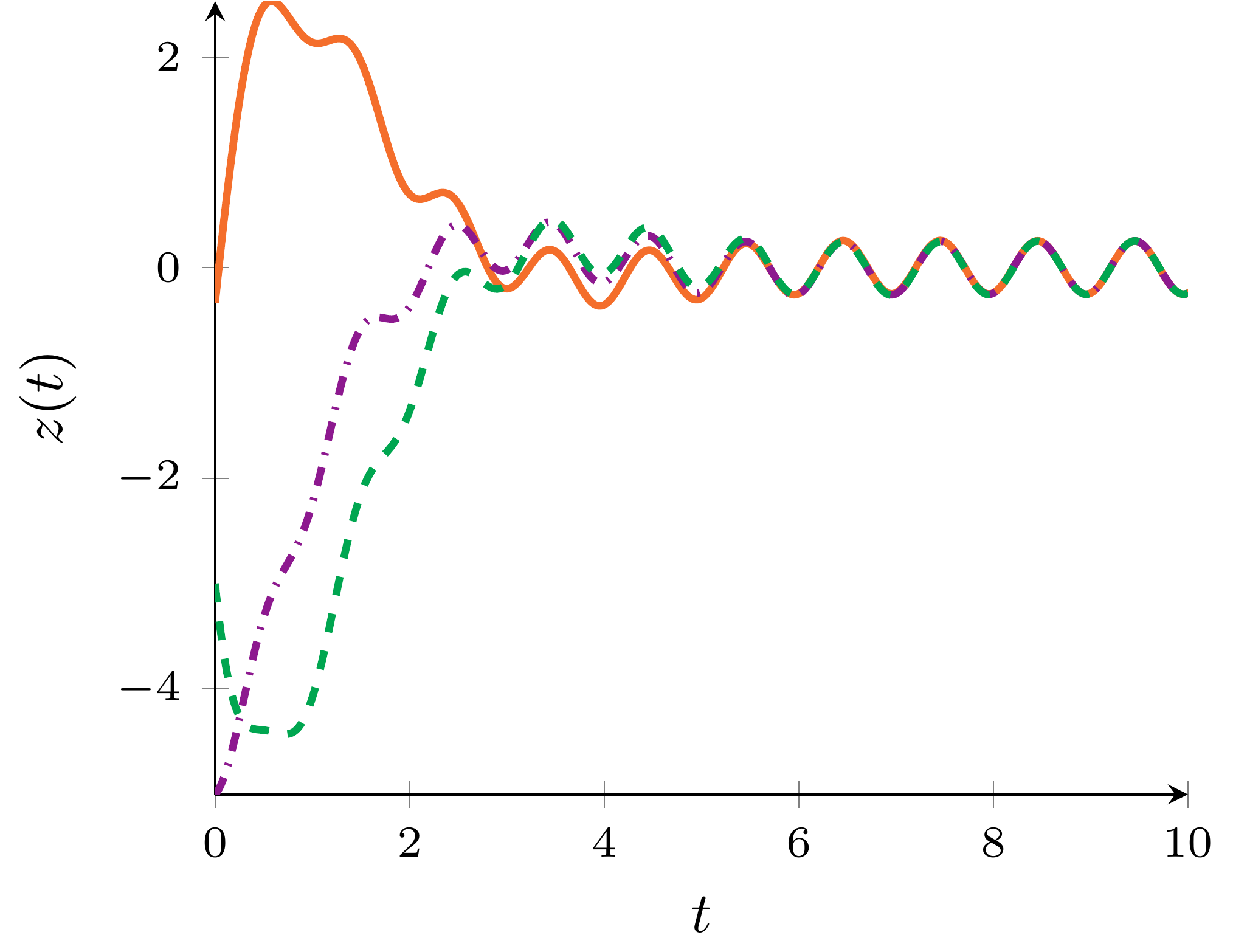}}\qquad
	\subcaptionbox{Position $\bar x(t)$ of the barycenter. \label{subfig:pos}}
	{\includegraphics[scale=0.2]{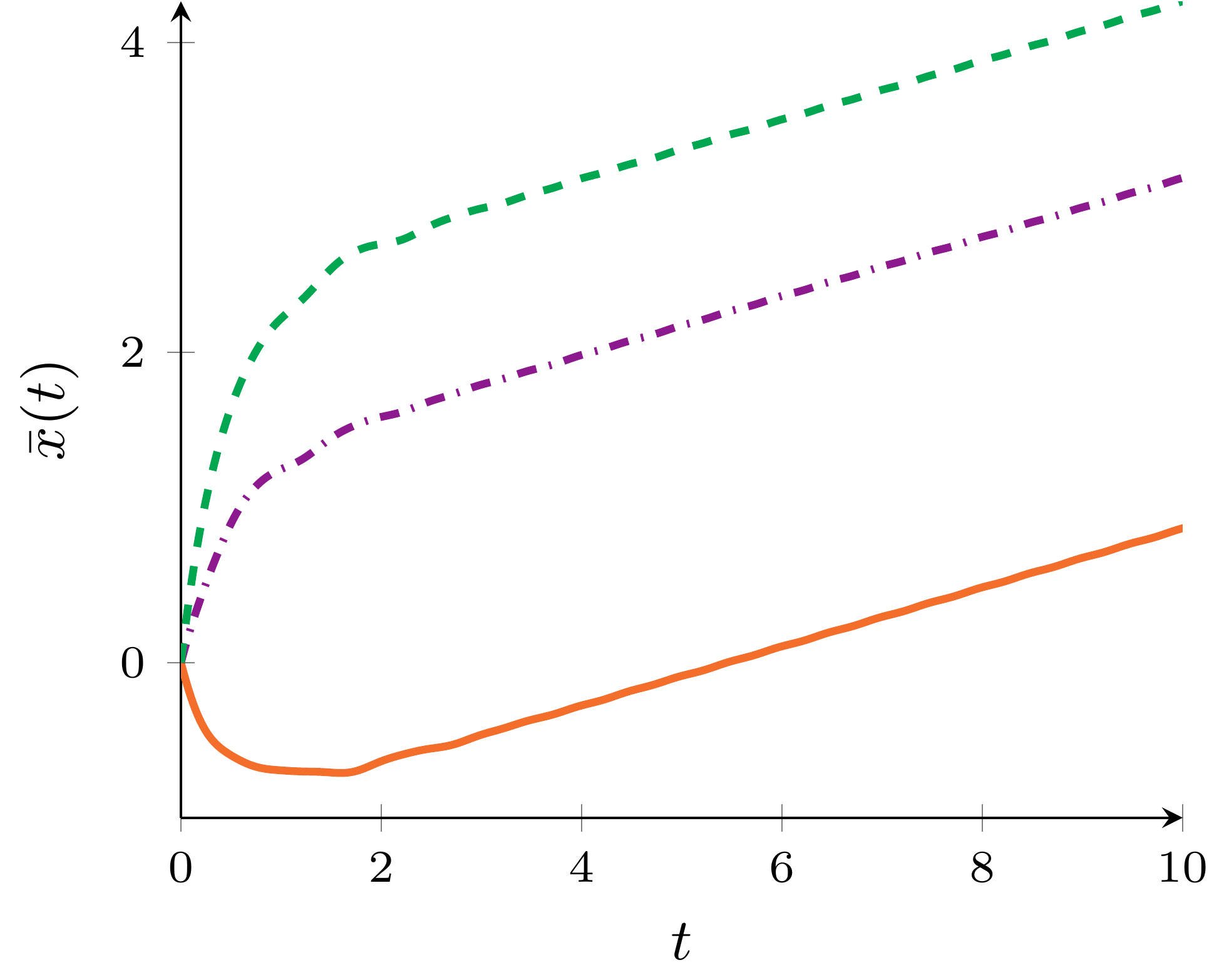}}
	\\[1cm]
	
	\subcaptionbox{Velocity $\dot{\bar x}(t)$ of the barycenter.\label{subfig:vel}}
	{\includegraphics[scale=0.2]{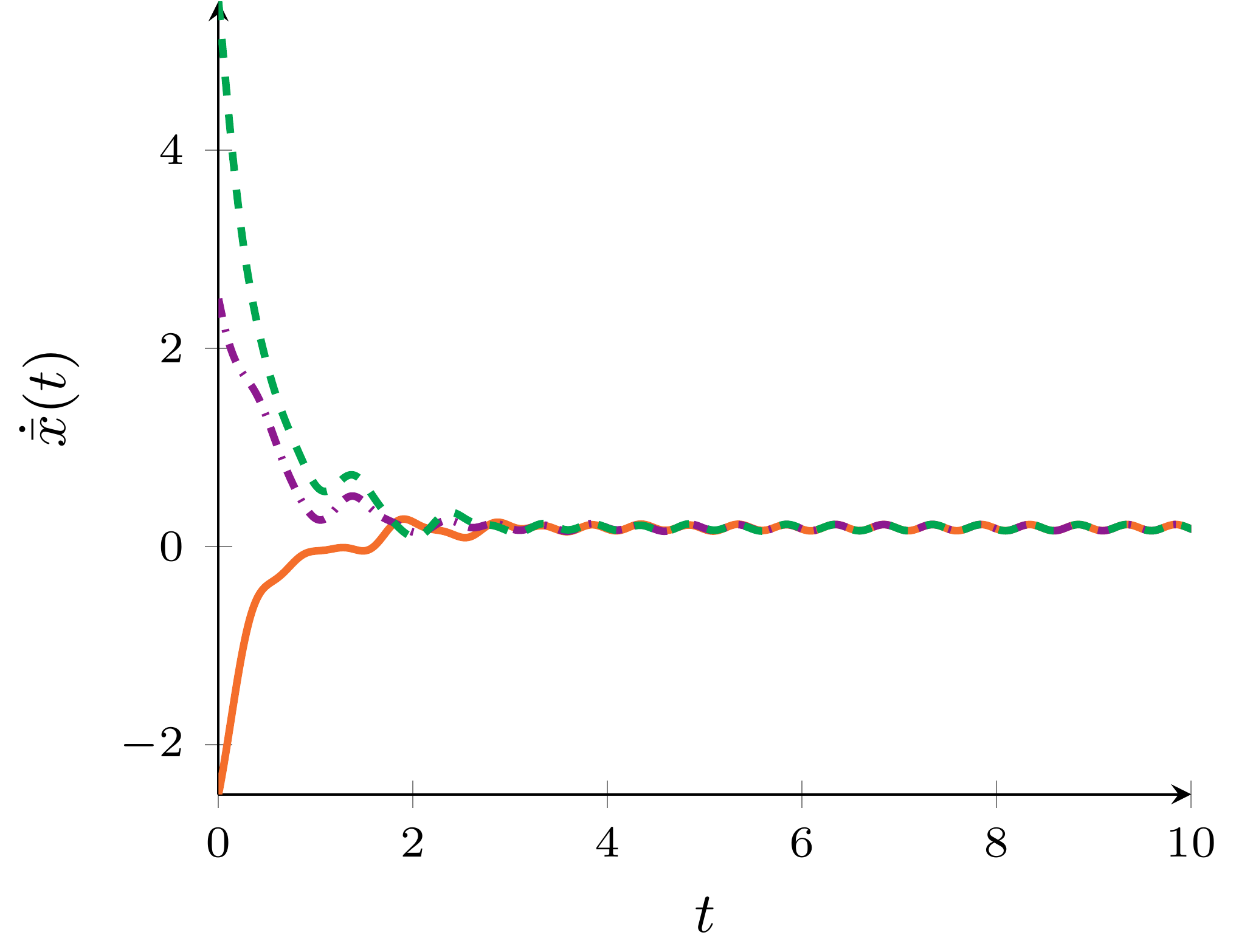}}\quad
	\subcaptionbox{Detail of the velocity $\dot{\bar x}(t)$ in figure \ref{subfig:vel}.\label{subfig:velzoom}}
	{\includegraphics[scale=0.2]{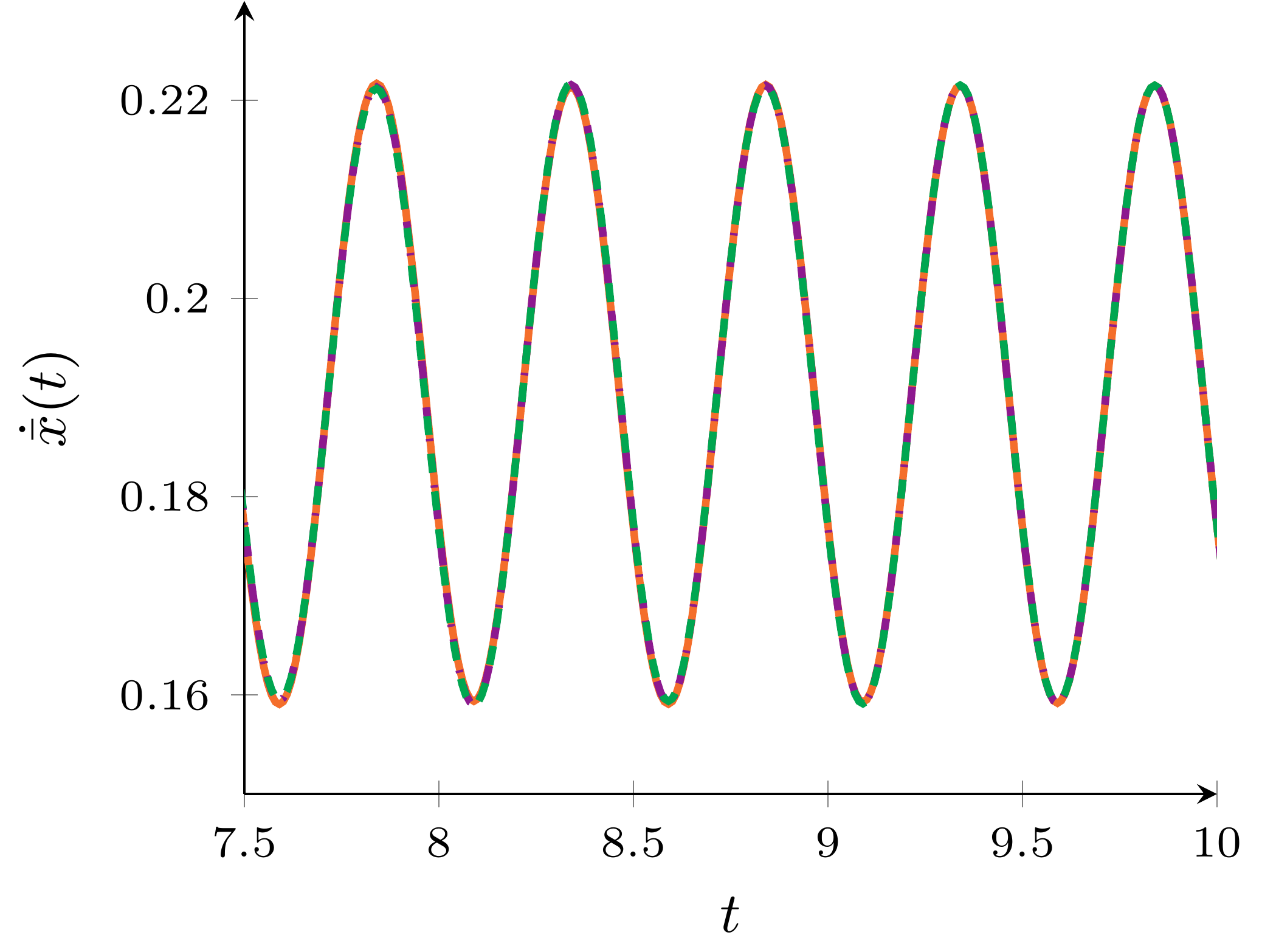}}
	\caption{Plots of the three solutions of System \eqref{eq:example} of Example~\ref{ex:linear}  with the following initial conditions: $(x_1(0),\dot{x}_1(0),  x_2(0),\dot{x}_2(0))=(\frac{1}{6},-7,-\frac{1}{6}, 2)$ for the orange, solid line,  $(x_1(0),\dot{x}_1(0),  x_2(0),\dot{x}_2(0))=(\frac{5}{2},2,-\frac{5}{2}, 3)$ for the violet, dashdotted line,  $(x_1(0),\dot{x}_1(0),  x_2(0),\dot{x}_2(0))=(\frac{3}{2},10,-\frac{3}{2}, 1)$ for the green, dashed line. }\label{fig:linear}
\end{figure}

\begin{example} \label{ex:linear}
Consider    equation \eqref{eqcrawl} with  $n=2$,  $m_1=m_2=1$, $\Ash= 1$, $\mu_1(t)=2+\sin 2\pi t,\,\mu_2(t)=2-\sin 2\pi t$, and  $L_1(t)=5\cos 2\pi t$. This leads to the system 
\begin{equation}\label{eq:example}
	\begin{cases}
		\ddot{x}_1=-(2+\sin 2\pi t)\dot{x}_1-x_1+x_2+5\cos 2\pi t\\
\ddot{x}_2=-(2-\sin 2\pi t)\dot{x}_2+x_1-x_2-5\cos 2\pi t 
	\end{cases}
\end{equation}
The simulation  of  system \eqref{eq:example} for three different initial value problems is presented in Figure~\ref{fig:linear}. We show the convergence for the shape (Fig.~\ref{subfig:shape}) and for the velocity of the barycenter (Fig.~\ref{subfig:vel}). Yet, favourable initial conditions might give a larger advancement of the position of the barycenter (Fig.~\ref{subfig:pos}): we therefore remark that all solutions converge to the same relative-periodic behaviour (i.e.~same limit in $(z,\dot x)$), but not necessarily to exactly the same relative-periodic solution, namely the difference between the position of the barycenters of two solutions might converge to a nonzero constant. 

We also observe that, while the shape converges to a $1$-periodic cycle, the velocity of the barycenter seems to converge to a $1/2$-periodic cycle (Fig.~\ref{subfig:velzoom}). This can be explained by noticing that the locomotion strategy proposed partially resembles the so-called \emph{inching} or \emph{two-anchor crawling} \cite{CaPiLa,Gid18}. During each actuation period, there are two phases when the system pushes itself forward: during elongation friction is greater on the rear block, which works partially as an anchor pushing the front block forward, while the opposite happens during contraction, with the front block anchoring and pulling forward the other one.

\end{example}

%%%%%%%%%%%%%%%%%%%%%%%%%%%%%%%%%%%%%%%%%%%%%%%%%%%%%%%%%%%%%

\section{Proof of Lemma~\ref{lemma:stiff_dissip}}\label{sec:stiff}

As a first step, we show that the shape $z(t)$ and its time derivative $\dot z(t)$ both enter eventually a compact set $\widetilde K$  independent of the initial values.
Applying $P$ to both sides of \eqref{eqcrawl} we get

\begin{align}
\ddot z=\begin{pmatrix}
	\ddot x_2-\ddot x_1\\
	\ddot x_3-\ddot x_2\\
	\vdots\\
	\ddot x_n -\ddot x_{n-1}
	\end{pmatrix}
&=
\frac{1}{\bar{m}}
\begin{pmatrix}
	f(t,\dot x_1)-f(t,\dot x_2)\\
f(t,\dot x_2)-f(t,\dot x_3)\\
\vdots\\
f(t,\dot x_{n-1})-f(t,\dot x_n)
\end{pmatrix}
-\frac{1}{\bar{m}}P P^\top\Ash (z-L(t)) \notag \\[2mm]
&\eqqcolon\Psi(t,\dot x)-\B (z-L(t)) \label{eq:z_second_order}
\end{align}
where we have set $\B\coloneqq \frac{1}{\bar{m}}P P^\top\Ash$.
We introduce now the  skewed energy function $\VV$ (see \cite{Hale2} for  applications of this technique to synchronization problems):
\begin{equation}
	\VV(z,\dot{z})= \frac{\norm{\dot{z}}^2}{2}+\frac{1}{2}\scal{\B z}{z}+\frac{\delta}{k} \scal{z}{\dot{z}}
\end{equation}
 where $\delta>0$ satisfies 
\begin{equation}\label{eq:delta_cond}
	0<\delta<\min\left\{\frac{2\hat \mu k}{(2\bar{m}+C_f)},\frac{k}{2},\frac{k^2}{\bar{m}}(1-\cos\frac{\pi}{n})\right\}
\end{equation}
The first upper bound  on $\delta$ will be used later, the other two provide sufficient conditions for $\VV$ to be a positive-definite quadratic form. In particular, notice  that the eigenvalues of $P P^\top$ are the values $2-2\cos\frac{j\pi}{n-1}$ for $j=1,\dots,n-1$ (see \cite[Example~7.4]{GrKa69}), so that the  smallest eigenvalue $\lambda_\mathrm{min}$ of $\B$ satisfies $\lambda_\mathrm{min}\geq 2\frac{k}{\bar{m}}(1-\cos\frac{\pi}{n-1})>0$.
 Hence, it holds
\begin{equation}\label{eq:Vcoerc}
	\lim_{\norm{(z,\theta)}\to+\infty}\VV(z,\theta)=+\infty
\end{equation} 
A computation of $\dot \VV\coloneqq\frac{\dd}{\dd t} \VV(z(t),\dot z(t))$, i.e., the derivative of $\VV$ along the orbits, gives
\begin{align*}
	\dot \VV(z,\dot z)&=\scal{\ddot z}{\dot z}+\scal{\B z}{\dot z}+\frac{\delta}{k}\norm{\dot{z}}^2+\frac{\delta}{k} \scal{z}{\ddot{z}}\\[1mm]
	&=	
\scal{\Psi(t,\dot x)-\B (z-L(t))}{\dot z}+\scal{\B z}{\dot z}+\frac{\delta}{k}\norm{\dot{z}}^2 +\frac{\delta}{k} \scal{z}{\Psi(t,\dot x)-\B (z-L(t))}	
\\[1mm]
&=\scal{\Psi(t,\dot x)}{\dot z}+\scal{\B L(t)}{\dot z} + \frac{\delta}{k}\scal{\B L(t)}{z}+\frac{\delta}{k}\norm{\dot{z}}^2+\frac{\delta}{k} \scal{z}{\Psi(t,\dot x)}-\frac{\delta}{k}\scal{\B z}{z}
\end{align*}
Let us estimate each term individually. By \ref{hyp:A4}, we have
\begin{equation*}
\scal{\Psi(t,\dot x)}{\dot z}=-\frac{1}{\bar{m}}\sum_{i=1}^{n-1}	(f(t,\dot x_{i+1})-f(t,\dot x_i))(\dot x_{i+1}-\dot x_i)\leq -\frac{\hat \mu}{\bar{m}} \sum_{i=1}^{n-1} \abs{\dot x_{i+1}-\dot x_i}^2 =-\frac{\hat \mu}{\bar{m}}\norm{\dot z}^2
\end{equation*}
The second and third terms are bounded in norm by a linear term $c_1(\norm{\dot z}+\norm{z})$ for some constant $c_1>0$.
The fourth will be controlled by the first one. For the fifth one, by Cauchy-Schwartz inequality we have
\begin{align*}
\scal{\Psi(t,\dot x)}{ z} &\leq\frac{1}{\bar{m}}\left(\sum_{i=1}^{n-1}\abs{z_i}^2\right)^\frac{1}{2} \left(\sum_{i=1}^{n-1}\abs{f(t,\dot x_{i+1})-f(t,\dot x_i)}^2\right)^\frac{1}{2}\leq\frac{\norm{z}}{\bar{m}}\left(\sum_{i=1}^{n-1}C_f^2 \abs{\dot z_i}^2\right)^\frac{1}{2}\\
&\leq \frac{C_f}{\bar{m}}\norm{z} \norm{\dot z}\leq \frac{C_f}{2\bar{m}}\left(\norm{z}^2+\norm{\dot z}^2\right)
\end{align*}
Thus, we have
\begin{equation} \label{eq:dotV}
	\dot \VV(z,\dot z)\leq \left[-\frac{\hat \mu}{\bar{m}}+\frac{\delta}{k}\left(1+\frac{C_f}{2\bar{m}}\right)\right]\norm{\dot z}^2 +\frac{\delta}{k}\left(\frac{C_f}{2\bar{m}}-\lambda_\mathrm{min}\right)\norm{z}^2 +c_1(\norm{\dot z}+ \norm{z})
\end{equation}
where we recall that $\lambda_\mathrm{min}>0$ is the smallest eigenvalue of $\B$. By assumptions \ref{hyp:stiffbody} and \eqref{eq:delta_cond} we have that the coefficients of the quadratic terms in the right-hand side of \eqref{eq:dotV} are both negative. 
Hence the derivative along the orbit in a generic point $(z,\dot z)$ satisfies
\begin{equation} \label{eq:Vsteep}
	\lim_{\norm{(z,\theta)}\to+\infty}\dot \VV(z,\theta)=-\infty
\end{equation} 

Let $K_1=\{(z,\theta)\in\R^{2(n-1)} \colon \dot \VV(z,\theta)\geq -1 \}$.  By \eqref{eq:Vsteep} we deduce that $K_1$ is bounded, and, by continuity of $\dot \VV$, compact. Let $M_1=\max_{(z,\theta)\in K_1}  \VV(z,\theta)$ and set $\widetilde K=\{(z,\theta) \colon  \VV(z,\theta)\leq M_1\}$. Since $K_1\subseteq \widetilde K$, if follows that $\widetilde{K}$ is forward invariant for the dynamics  defined by \eqref{eq:z_second_order} on $\R^{2(n-1)}$. Moreover, for any given initial value \eqref{eq:initialvalue}, setting
\begin{equation*}
	\tilde t_0\coloneqq\begin{cases}
		0 &\text{if $\VV(z_0,P v_0)\leq M_1$}\\
		\VV(z_0,P v_0)-M_1 &\text{if $\VV(z_0,P v_0)>M_1$}\\
	\end{cases}	
\end{equation*}
we observe that the solution of \eqref{eq:z_second_order} with initial conditions $z(0)=z_0$ and $\dot z(0)=P v_0$ satisfies	$(z(t),\dot z(t))\in \widetilde K$ for every $t\geq \tilde t_0$.

To conclude the proof, let us introduce the center of mass of the system $\bar x=\frac{1}{M}\sum_{i=1}^{n}m_ix_i$. As seen in Section~\ref{sec:viscous}, writing $u=\dot{\bar x}$ and $\alpha_i(t)=x_i(t)-\bar x(t)$, we have
\begin{equation} \label{eq:1card_bis}
\dot u= -\frac{1}{n\bar{m}}\sum_{i=1}^n f(t,u+\dot\alpha_i) \qquad \text{a.e. $t\in[0, +\infty)$.}
\end{equation}
Since, by \eqref{eq:baric_decomp}, the $\dot\alpha_i$ are (constant) linear combinations of the $\dot z_i$, by the boundedness of the $\dot z_i$ we know that there exists a constant $c_2>0$ such that, for every initial value problem \eqref{eq:initialvalue}, it holds $\abs{\dot{\alpha}_i(t)}<c_2$ for every $t>\tilde t_0$ and $i=1,\dots,n$.

We now prove that, for every initial value problem \eqref{eq:initialvalue}, there exists a time $t_0>\tilde t_0$, depending on $(z_0,v_0)$, such that the corresponding solution satisfies\begin{equation}\label{eq:baric_bound}
\abs{u(t)}=\abs{\dot{\bar{x}}(t)}\leq c_2+1 \qquad\text{for  every $t\geq t_0$}. 
\end{equation}
By \ref{hyp:A4} and \eqref{eq:1card_bis} we deduce that for almost every  $t>\tilde t_0$ it holds
\begin{equation}\label{eq:baric_dissip}
%&\frac{\dd \abs{u(t)}}{\dd t}<0 &&\text{if $\abs{u(t)}\geq c_2$}\\
\frac{\dd \abs{u(t)}}{\dd t}<-\frac{\hat \mu}{\bar{m}}\qquad\text{when $\abs{u(t)}\geq c_2+1$}
\end{equation}
By \eqref{eq:baric_dissip} we deduce that if $\abs{u(t)}\leq c_2+1$, then $\abs{u(s)}\leq c_2+1$ for every $s\geq t$. Moreover, by \eqref{eq:baric_dissip} we also obtain that if $\abs{u(t)}> c_2+1$, then \begin{equation*}
\abs{u(t+\beta)}\leq c_2+1\qquad\text{where $\beta=\frac{\bar{m}(c_2+1-\abs{u(t)}) }{\hat\mu}>0$}
\end{equation*}
Thus, setting
\begin{equation*}
t_0\coloneqq\begin{cases}
\tilde t_0 &\text{if $\abs{u(\tilde t_0)}\leq c_2+1$}\\[1mm]
\tilde t_0+\displaystyle\frac{\bar{m} (c_2+1-\abs{u(\tilde t_0))} }{\hat\mu} &\text{if $\abs{u(\tilde t_0)}> c_2+1$}
\end{cases}	
\end{equation*}
it follows that \eqref{eq:baric_bound} is satisfied and
\begin{equation*}
(z(t),\dot z(t),\dot{\bar x}(t))\in \widetilde{K}\times [- c_2-1, c_2+1]\subset \R^{2n-1} \qquad\text{for every $t\geq t_0$}
\end{equation*}
Since  $\widetilde{K}\times [- c_2-1, c_2+1]$ is compact and  $(z,\dot z,\dot{\bar x})$ is a (time-independent) linear transformation of $(z,v)$ in $\R^{2n-1}$, the result follows.
\qed

\paragraph{Acknowledgements.}  Paolo Gidoni is a member of the Gruppo Nazionale di Fisica Matematica of the
Istituto Nazionale di Alta Matematica.
Alessandro Margheri  was  supported  by FCT project  UIDB/04561/2020: https://doi.org/10.54499/UIDB/04561/2020~. The Authors are grateful to Filippo Riva for a valuable discussion concerning the proof of Theorem~\ref{th:glob_existence}.

\appendix 
\section{Proof of Theorem \ref{th:glob_existence}} \label{sec:globalexistence}

\begin{proof}
Fix $t_0\in\R$ e  let  $t\to  x(t)$  be a solution of \eqref{eqcrawl}  defined on the right maximal interval $J\coloneqq [t_0, t_0+\omega).$   We argue by contradiction and assume that $\omega<+\infty$.
Taking the scalar product of \eqref{eqcrawl}  with  $\dot x(t)$ and then  integrating the result between $t_0$ and $t$ we get 
\begin{equation*}
\frac{1}{2}\norm{\dot x(t)}^2_{\mathbb{M}}+\frac{1}{2}\langle\A x(t), x(t)\rangle +\int_{t_0}^t \langle F(s,\dot x(s)), \dot x(s)\rangle\dd s=C_0+\int_{t_0}^t \langle\ell(s), \dot x(s)\rangle\dd s
\end{equation*}
for every $t\in J$, where we   define  
$C_0\coloneqq \frac{1}{2}\norm{\dot x(t_0)}_{\mathbb{M}}^2+\frac{1}{2}\langle\A x(t_0), x(t_0)\rangle$. Since by assumption \ref{hyp:Dfric}  we have $\langle F(s,\dot x(s)), \dot x(s)\rangle\geq 0$  and $\A$ is positive semi-definite, using Cauchy-Schwartz inequality and the equivalence of norms in $\R^n$,  we obtain
\begin{equation*}
\frac{1}{2}\norm{\dot x(t)}^2_{\mathbb{M}}\leq C_0+C^*\norm{\ell}_{L^\infty}\int_{t_0}^t \norm{\dot x(s)}_{\mathbb{M}}\dd s,
\end{equation*}
where $C^*>0$  is a suitable constant.  Writing $\widehat{C}\coloneqq C^*\norm{\ell}_{L^\infty}$, by  Young's inequality we then obtain  
\begin{equation*}
\frac{1}{2}\norm{\dot x(t)}^2_{\mathbb{M}}\leq C_0+\frac{\widehat{C}}{2} (t-t_0)+\widehat{C}\int_{t_0}^t \frac{1}{2}\norm{\dot x(s)}^2_{\mathbb{M}}\dd s
\end{equation*}
for $t\in J.$  Since $\alpha(t)\coloneqq C_0+\frac{\widehat{C}}{2} (t-t_0)$ is increasing in $J$, by Gronwall's Lemma (see e.g.~\cite[Theorem 1.1]{Hart})  it follows that
\begin{equation*}
	\frac{1}{2}\norm{\dot x(t)}^2_{\mathbb{M}}\leq \alpha(\omega) e^{\widehat{C}(t-t_0)}\leq \alpha(\omega)e^{\widehat{C}(\omega-t_0)}\in \R  \quad \text{for every $t\in J$.} 
\end{equation*}

Then, by the general theory of ordinary differential equations, the solution $x(t)$ can be extended beyond $\omega$, contradicting our initial assumption that $[t_0,\omega)$ is the right maximal interval of existence. Therefore, we conclude that $\omega = +\infty$.
\end{proof}


\begin{thebibliography}{40} \setlength{\itemsep}{-0.5mm} \footnotesize
\frenchspacing
\bibitem{ADGZ17}  F. Alouges, A. DeSimone, L. Giraldi, and M. Zoppello,  Purcell Magneto-Elastic Swimmer Controlled by an External Magnetic Field. IFAC-PapersOnLine 50 (2017), 4120--4125.
\bibitem{BPZZ16} N. Bolotnik, M. Pivovarov, I, Zeidis, and K Zimmermann, 
On the motion of lumped-mass and distributed-mass self-propelling systems in a linear resistive environment. ZAMM Z. Angew. Math. Mech. 96 (2016), 747--757.
\bibitem{CaPiLa} M. Calisti, G. Picardi, and C. Laschi, Fundamentals of soft robot locomotion. J. Roy. Soc. Interf. 14 (2017), 20170101. 
\bibitem{Capistrano2023}  R.A. Capistrano-Filho and I.M. de Jesus, Massera’s Theorems for a Higher Order Dispersive System. Acta Appl. Math. 185 (2023), art. 5.
\bibitem{ChowHale}  S.N. Chow and J.K. Hale, Strongly limit-compact maps, Funkcial. Ekvac. 17 (1974), 31--38.
\bibitem{ColGidVil}  G. Colombo, P. Gidoni and E. Vilches, Stabilization of periodic sweeping processes and asymptotic average velocity for soft locomotors with dry friction,  Discrete Contin. Dyn. Syst.  42 (2022), 737--757.
\bibitem{DeSTat12} A. DeSimone and A. Tatone, Crawling motility through the analysis of model locomotors: two case studies. Eur. Phys. J. E. 35 (2012), art. 85.
\bibitem{EldJac}  J. Eldering and H.O. Jacobs, The role of symmetry and dissipation in biolocomotion, SIAM J. Appl. Dyn. Syst. 15 (2016), 24--59.
\bibitem{FaPaZo}  F. Fassò, S. Passarella, and M. Zoppello,
Control of locomotion systems and dynamics in relative periodic orbits,
J. Geom. Mech. 12 (2020), 395--420.
\bibitem{FigKny} T. Figurina and D. Knyazkov,
Periodic gaits of a locomotion system of interacting bodies,
Meccanica 57 (2022), 1463--1476. 
\bibitem{FlMeSl}  M. Fleury, J. G. Mesquita and A. Slavík,
Massera's theorems for various types of equations with discontinuous solutions,
Journal of Differential Equations 269 (2020), 11667--11693.
\bibitem{Gid18} P. Gidoni, Rate-independent soft crawlers, Quart. J. Mech. Appl. Math. 71 (2018), 369--409.
\bibitem{GiMaRe} P. Gidoni, A. Margheri and C. Rebelo  Limit cycles for dynamic crawling locomotors with periodic prescribed shape, Z. Angew. Math. Phys. (2023) 74--46
\bibitem{GidRiv} P. Gidoni and F. Riva, A vanishing inertia analysis for finite dimensional rate-independent systems with nonautonomous dissipation and an application to soft crawlers, Calc. Var. Partial Differential Equations,  60 (2021), art. 191.
\bibitem{GrKa69} R. T. Gregory and D. Karney, A collection of matrices for testing computational algorithm, Wiley-Interscience, 1969.
\bibitem{GudMak} I. Gudoshnikov and O. Makarenkov, Stabilization of the response of cyclically loaded lattice spring models with plasticity, ESAIM: Control, Optimization and Calculus of Variations 27 (2021), paper S8. 
\bibitem{Hale} J. Hale,  Dissipation and compact attractors, Journal of Dynamics and Differential Equations 18 (2006), 485--523. 
\bibitem{Hale2} J. Hale, Diffusive Coupling, Dissipation and Synchronization, Journal of Dynamics and Differential Equations, Vol. 9,No 1, 1997, 1--52.
\bibitem{Hale80} J. Hale, Ordinary differential Equations, R.E. Krieger, Huntington, New York, 1980.
\bibitem{Hart}  Hartman P: Ordinary Differential Equations. John Wiley \& Sons, New York, 1964.
\bibitem{KelMur} S.D. Kelly and R.M. Murray, Geometric phases and robotic locomotion. J. Robot. Syst. 12 (1995), 417--431.
\bibitem{KropacekAl} J.~Kropacek, C.~Maslen, P.~Gidoni, P.~Cigler, F.~Stepanek and I.~Rehor, Light-Responsive Hydrogel Microcrawlers, Powered and Steered with Spatially Homogeneous Illumination, Soft Robotics 11 (2024), 531--538. 
\bibitem{Levillain24} J. Levillain, F. Alouges, A. Desimone, A. Choudhary, S. Nambiar, and I. Bochert, A bi-directional low-Reynolds-number swimmer with passive elastic arms, \texttt{arXiv:2403.10556}\,.
\bibitem{levinson44} N. Levinson, Transformation theory of nonlinear differential equations of the second order. Ann. Math. 45 (1944), 724--737.
\bibitem{LCLL99}  Y. Li, F. Cong, Z. Lin and W. Liu, Periodic solutions for evolution equations, Nonlinear Anal. 36 (1999), 275--293.
\bibitem{Massera} J.L. Massera, The existence of periodic solutions of systems of differential  equations,   Duke Math. J. (1950), 457--475.
\bibitem{Mak} O. Makarenkov, Existence and stability of limit cycles in the model of a planar passive biped walking down a slope, Proc. R. Soc. A. 476 (2020), 20190450.
\bibitem{MonDeS} A. Montino and A. DeSimone, Three-sphere low-Reynolds-number swimmer with a passive elastic arm. Eur. Phys J E 38 (2015), art. 42.
\bibitem{Or} R. Ortega,  Periodic differential equations in the plane, a topological perspective. Series in Nonlinear Analysis and Applications, 29, De Gruyter (2019).
\bibitem{PasOr} E. Pasov and Y. Or, Dynamics of Purcell’s three-link microswimmer with a passive elastic tail. Eur. Phys J E 35 (2012), art. 78.
\bibitem{Pliss} V. Pliss,  Nonlocal Problems in the Theory of Oscillations,  Academic Press, New York, 1966. 
\bibitem{PoFeTa} B. Pollard, V. Fedonyuk, and P. Tallapragada, Swimming on limit cycles with nonholonomic constraints, Nonlinear Dyn 97 (2019), 2453--2468.
\bibitem{RehorAl21} I.~Rehor, C.~Maslen, P.G.~Moerman, B.G.P.~Van Ravensteijn, R.~Van Alst, J.~Groenewold, H.B.~Eral and W.K. Kegel, Photoresponsive hydrogel microcrawlers exploit friction hysteresis to crawl by reciprocal actuation, Soft Robotics 8 (2021), 10--18.
\bibitem{RouHaLa}  N. Rouche, P. Habets and M. Laloy, Stability Theory by Liapunov's Direct Method,  Applied Mathematical Sciences book series, Springer-Verlag New York 1977. 
\bibitem{SanZop}N. Sansonetto and M. Zoppello, On the trajectory generation of the hydrodynamic Chaplygin sleigh. IEEE Control Systems Letters 4 (2020), 922--927.
\bibitem{Sell} G.R. Sell,  Periodic Solutions and Asymptotic stability,  Journal of Differential Equations  2 (1966), 143--157
\bibitem{Smith}  R. A. Smith, Massera's convergence Theorem for periodic nonlinear differential equations , Journal of Mathematical analysis and applications 120 (1986),  679--708.
\bibitem{ShNa}  J.S. Shin  and T. Naito, Semi–Fredholm operators and periodic solutions for linear functional differential equations, J. Differential Equations 153 (1999), 407--441.
\bibitem{TINK12} Y. Tanaka, K. Ito, T. Nakagaki, and R. Kobayashi, Mechanics of peristaltic locomotion and role of anchoring, J. R. Soc. Interface 9 (2012), 222--233.  
\bibitem{WagLau}  G.~L.~Wagner and E.~Lauga,  Crawling scallop: friction-based locomotion with one degree of freedom, J. Theoret. Biol. 324 (2013), 42--51.

\end{thebibliography}
\end{document}